\documentclass[12pt,a4paper]{amsart}
\makeatletter
\renewcommand\normalsize{%
    \@setfontsize\normalsize{11.7}{14pt plus .3pt minus .3pt}%
    \abovedisplayskip 10\p@ \@plus4\p@ \@minus4\p@
    \abovedisplayshortskip 6\p@ \@plus2\p@
    \belowdisplayshortskip 6\p@ \@plus2\p@
    \belowdisplayskip \abovedisplayskip}
\renewcommand\small{%
    \@setfontsize\small{9.5}{12\p@ plus .2\p@ minus .2\p@}%
    \abovedisplayskip 8.5\p@ \@plus4\p@ \@minus1\p@
    \belowdisplayskip \abovedisplayskip
    \abovedisplayshortskip \abovedisplayskip
    \belowdisplayshortskip \abovedisplayskip}
\renewcommand\footnotesize{%
    \@setfontsize\footnotesize{8.5}{9.25\p@ plus .1pt minus .1pt}
    \abovedisplayskip 6\p@ \@plus4\p@ \@minus1\p@
    \belowdisplayskip \abovedisplayskip
    \abovedisplayshortskip \abovedisplayskip
    \belowdisplayshortskip \abovedisplayskip}
\ifdefined\pdfpagewidth
\else
\fi
\calclayout
\makeatother

\usepackage{amsmath}
\usepackage{amsthm}
\usepackage{amssymb}
\usepackage{enumitem}
\usepackage{pdfpages}
\usepackage{hyperref}
\usepackage[capitalize]{cleveref}

\newcounter{alph}

\newtheorem{theo}[alph]{Theorem}
\newtheorem{propo}[alph]{Proposition}
\newtheorem{coro}[alph]{Corollary}
\numberwithin{equation}{section}
\newtheorem{cor}[equation]{Corollary}
\newtheorem{lem}[equation]{Lemma}

\newtheorem{prop}[equation]{Proposition}
\newtheorem{thm}[equation]{Theorem}
\theoremstyle{definition}

\newtheorem{dfn}[equation]{Definition}
\newtheorem{exa}[equation]{Example}
\newtheorem{rem}[equation]{Remark}

\newtheorem{prob}[alph]{Problem}

\DeclareMathOperator{\ad}{ad}
\DeclareMathOperator{\Ad}{Ad}
\DeclareMathOperator{\cus}{cus}

\DeclareMathOperator{\diam}{diam}

\DeclareMathOperator{\geo}{geo}
\DeclareMathOperator{\Hdim}{dim_{Hd}}

\DeclareMathOperator{\rad}{rad}

\DeclareMathOperator{\Sl}{SL}
\DeclareMathOperator{\supp}{supp}

\DeclareMathOperator{\insize}{insize}

\def\C{\mathbb C}
\def\E{\mathbb E}
\def\F{\mathbb F}
\def\H{\mathbb H}
\def\K{\mathbb K}

\def\N{\mathbb N}
\def\O{\mathbb O}
\def\p{\mathbb P}
\def\R{\mathbb R}
\def\Z{\mathbb Z}

\def\ve{\varepsilon}
\def\vf{\varphi}
\def\la{\langle}
\def\ra{\rangle}

\begin{document}

\title[Martin boundaries of diffusions and random walks]{Martin boundaries of diffusions and random walks on hyperbolic spaces}
\author{Werner Ballmann}
\address{WB: Max Planck Institute for Mathematics, Vivatsgasse 7, 53111 Bonn}
\email{hwbllmnn\@@mpim-bonn.mpg.de}
\author{Debanjan Nandi}
\address{DN: Stat-Math unit, Indian Statistical Institute, 203 B.T. Road, Kolkata-700108}
\email{debanjannandi42@gmail.com}
\author{Panagiotis Polymerakis}
\address{PP: Department of Mathematics, University of Thessaly, 3rd km Old National Road Lamia–Athens, 35100, Lamia, Greece}
\email{ppolymerakis\@@uth.gr}

\begin{abstract}
We discuss certain kinds of diffusions on hyperbolic spaces, associated random walks on discrete groups of iso\-metries of the latter, and their Martin boundaries.    
\end{abstract}

\date{September 17, 2025}


\thanks{\emph{Acknowledgments.}
We thank Anna Erschler for pointing out to the third author that points in the Martin boundaries of random walks may be super-harmonic, but not harmonic.
Her note initiated the work on \cite{BPe} and on this paper.
We are grateful to the Max Planck Institute for Mathematics
and the Hausdorff Center for Mathematics in Bonn for their support and hospitality. 
D.\,N.\;gratefully acknowledges INSPIRE faculty fellowship DST/INSPIRE/04/2023/001952.}
\maketitle


\setcounter{tocdepth}{1}
\tableofcontents

\section*{Introduction}

Martin boundaries are asymptotic objects associated to transient random processes on topological spaces.
In contrast to Poisson boundaries, which are measure theoretic objects, Martin boundaries are of a topological nature.
They are quite easy to define, but difficult to determine.
In general, they do not seem to be stable, which may be part of the problem behind their determination.

Spaces of interest in this article are \emph{open Riemannian manifolds $H$},
that is, non-compact, complete, and connected Riemannian manifolds,
and discrete subsets $X\subseteq H$, as a rule invariant under a properly discontinuous group $\Gamma$ of isometries of $H$.
Specific examples are \emph{Hadamard manifolds}, that is, complete and simply connected Riemannian manifolds of non-positive (sectional) curvature, where $\Gamma$ is the group of covering transformations for a subcover $M=\Gamma\backslash H$ and $X$ is the preimage of a point or a discrete subset of $M$.
Frequently, we assume that the curvature of $H$ is pinched between negative bounds.
The most prominent examples are the hyperbolic spaces $H=H_\F^k$, where $\F\in\{\R,\C,\H,\O\}$.
Our focus is on relations between the geometric boundary $\partial_{\geo}H$ of $H$ and the Martin boundaries $\partial_{\mathcal{S}}H$  and $\partial_{\mu}X$ of certain transient diffusions $\mathcal{S}$ on $H$ and random walks $\mu$ on $X$, specifically those, where $\mathcal{S}$ and $\mu$ are related by the Lyons-Sullivan discretization procedure.
For definitions, we refer the reader to Sections \ref{secmark} and \ref{ssecls}.

If $H$ is a Hadamard manifold with pinched negative curvature, then Brownian motion $\mathcal B$ on $H$ is transient.
It was shown by Anderson and Schoen \cite{AS}, that the identity of $H$ extends continuously to a homeomorphism between the geometric compactification $\bar H=H\cup\partial_{\geo}H$ and the Martin compactification $\hat H=H\cup\partial_{\mathcal B}H$ of $H$.
Different approaches to this were given by Ancona and Kifer \cite{An,Ki}.

Recall that Brownian motion is the diffusion associated with the Laplace operator $\Delta=\nabla^*\nabla$ (on functions).
Ancona considered also more general operators, among them diffusion operators of the form
\begin{align}\label{diffoper}
    D = \Delta - 2\nabla\ln\vf,
\end{align}
where $\vf$ is a positive $\lambda$-harmonic function.
The diffusion $\mathcal{S}$ associated to such an operator $D$ plays a prominent role in Sullivan's work on positivity in Riemannian geometry \cite{Su3};
we call $\mathcal{S}$ \emph{Sullivan's $\vf$-process} or \emph{$\vf$-process}.
It is complete if and only if, for all $x\in H$ and $t>0$,
\begin{align}\label{complete}
	\vf(x) = \int_H e^{\lambda t} p(t,x,y) \vf(y) d\mathcal{L}(y),
\end{align}
where $p=p(t,x,y)$ denotes the fundamental solution of the heat equation and $\mathcal{L}$ Riemannian volume.
Completeness of $\vf$ implies that the measure $\vf^2\mathcal{L}$ is invariant under the $\vf$-process.
The transition kernel of the $\vf$-process is given by 
\begin{align}\label{sullivanpt}
    e^{\lambda t} p(t,x,y)\vf(y)/\vf(x).
\end{align}

Let  $\lambda_0(H)$ be the bottom of the $L^2$-spectrum of the Laplacian on $H$.
Recall that $\lambda_0(H)$ is the maximal $\lambda$, such that there are positive $\lambda$-harmonic functions on $H$.
For each $\lambda<\lambda_0(H)$, the bounded operator $(\Delta-\lambda)^{-1}$ on $L^2(H)$ is given by the \emph{Green kernel}
\begin{align}\label{greenker}
	G_\lambda(x,y) = \int_0^\infty e^{\lambda t}p(t,x,y)dt.
\end{align}
For all $x\ne y$ in $H$, the integral in \eqref{greenker} converges for all $\lambda<\lambda_0(H)$ and diverges for all $\lambda>\lambda_0(H)$.
There is a dichotomy at $\lambda=\lambda_0(H)$: Either the integral converges for all $x\ne y$ or it diverges for all such pairs \cite[Theorem 2.6]{Su3}.
Following Sullivan, we say that $H$ is \emph{$\lambda_0$-recurrent}, if it diverges.
Then all positive $\lambda_0(H)$-harmonic functions on $H$ are constant multiples of one another \cite[Theorem 2.7]{Su3}.
Moreover, if $H$ carries a square-integrable positive $\lambda_0(H)$-harmonic function, then $H$ is $\lambda_0$-recurrent.
If $H$ is $\lambda_0$-recurrent and $\vf$ is a positive $\lambda_0(H)$-harmonic function, then the $\vf$-process is complete, recurrent, and preserves the measure $\vf^2\mathcal L$ \cite[Theorem 2.10]{Su3}.
If $\vf$ is square-integrable, then the latter measure is finite.
By what we said above, the $\vf$-process does not depend on the choice of $\vf$ (for $\lambda=\lambda_0(H)$).
Clearly, if $\lambda_0(H)=0$, then $\vf$ is constant and $\mathcal S$ equals Brownian motion.

The analogous results hold for $M=\Gamma\backslash H$.
Since lifts of positive $\lambda$-harmonic functions on $M$ are positive $\lambda$-harmonic functions on $H$,
we have $\lambda_0(M)\le\lambda_0(H)$ by the above characterization of $\lambda_0$ by positive $\lambda$-harmonic functions.
If $\vf$ is the lift of a positive $\lambda$-harmonic function $\psi$ on $M$ to $H$, then the $\vf$-process on $H$ is the lift of the $\psi$-process on $M$.
In particular, if $M$ is $\lambda_0$-recurrent (where here $\lambda_0=\lambda_0(M)\le\lambda_0(H)$ by our terminology) and $\psi$ is a positive $\lambda_0(M)$-harmonic function on $M$, then $\Gamma$-invariant positive $\lambda$-harmonic functions on $H$ are constant multiples of the lift $\vf$ of $\psi$ to $H$.
The associated $\Gamma$-invariant $\vf$-process is complete and preserves the measure $\vf^2\mathcal{L}$ on $H$.

Let $\mathcal{S}$ be a complete and transient $\vf$-process on $H$.
Let $X$ be a $\Gamma$-invariant discrete subset of $H$ and $\mu=(\mu_x)_{x\in X}$ be a $\Gamma$-equivariant family of probability measures on $X$.
Assume that the associated random walk on $X$, also denoted  by $\mu$, is irreducible and transient.
Axiomatizing the corresponding notion from \cite{BP}, we say that $\mu$ is an \emph{FLS-restriction} of $\mathcal S$ to $X$ if the Martin kernels $K$ of $\mathcal{S}$ and $k$ of $\mu$ with origin some $o\in X$ satisfy $k(x,y)=K(x,y)$ for all $x,y\in X$ with $y\neq x, o$.
Here FLS stands for Furstenberg \cite{Fu}, Lyons, and Sullivan \cite{LS}.

To state our first result, we need some further terminology.
For a Hada\-mard manifold $H$ of pinched negative curvature and a subset $X$ of $H$, we denote by $\Lambda_X$ the set of limit points of $X$ in $\partial_{\geo}H$.
If $X=\Gamma x$ is an orbit of $\Gamma$, we write $\Lambda_\Gamma$ for $\Lambda_{\Gamma x}$ since the latter does not depend on the choice of $x$ and call $\Lambda_\Gamma$ the \emph{limit set of $\Gamma$}.

\begin{theo}\label{main1}
Let $M=\Gamma\backslash H$ be $\lambda_0$-recurrent and $X$ an orbit of $\Gamma$.\\
1) Assume that the associated $\Gamma$-invariant $\vf$-process $\mathcal{S}$ on $H$ is transient.
Then $X$ admits an irreducible symmetric and $\Gamma$-invariant FLS-restriction $\mu$\ of $\mathcal S$ to $X$.
If the Ricci curvature of $H$ is bounded from below and $M$ admits a square-integrable $\lambda_0(M)$-eigenfunc\-tion, then $\mu$ can be chosen to have finite first moment and finite entropy. \\
2) Assume that both of the following two conditions are satisfied, \\
$\phantom{xx}$ i)
 $H$ is a Hadamard manifold with pinched negative curvature, \\
$\phantom{xx}$ ii)
$\lambda_0(M)<\lambda_0(H)$ or $H$ is a symmetric space. \\
Then $\mathcal{S}$ is transient and any FLS-restriction $\mu$ to $X$ induces a homeomorphism $\Lambda_\Gamma\to\partial_\mu X$, which extends the identity of $X$ continuously.
\end{theo}

The existence of FLS-restrictions is obtained by a refined version of the discretization procedure of Brownian motion by Lyons and Sullivan; see \cite{LS} and \cite[Theorem 3.1]{BL2}.
The name of Furstenberg occurs here because he was the first to obtain discretizations of Brownian motion, in his case for lattices in $\Sl(n,\R)$; see \cite[Sections 5 and 6]{Fu}.
The last assertion of \cref{main1} may be viewed as a stability result for Martin boundaries.

\cref{main1} extends \cite[Theorem 3.2]{BL2} and follows from our Theorems \ref{Martin} and \ref{fm} in the body of the text.
Note that the last step of the proof of finiteness of first moment in \cite{BL2} is incomplete.
Assuming geometrical finiteness in the sense of Bowditch \cite{Bo}, the assertion about the finiteness of first moment will be improved, in \cite{Nan2}, to finiteness of exponential moments. 

The \emph{Poincar\'e series} of $\Gamma$ (as group of isometries of $H$) is given by
\begin{align}
    P(x,y,s) = \sum_\Gamma e^{-s\rho(x,gy)}
\end{align}
The \emph{critical exponent} of $\Gamma$ is defined to be the real number $\delta(\Gamma)$ such that $P(x,y,s)$ converges for $s>\delta(\Gamma)$ and diverges for $s<\delta(\Gamma)$, where $x,y$ is some--or any--pair of points in $H$.
At $s=\delta(\Gamma)$, $P(x,y,s)$ might converge or diverge.
We say that $\Gamma$ is of \emph{divergence type} if it diverges.

Let now $H=H_\F^k$ be the hyperbolic space over $\F$ as above with dimension $m=kd$, $d=\dim_{\R}\F$, endowed with the standard (Fubini-Study) metric with maximal curvature $-1$.
Recall that $\lambda_0(H)=h^2/4$, where
\begin{align}\label{volentin}
    h = m - d + 2(d-1) = m + d - 2 > 0
\end{align}
denotes the asymptotic rate of volume growth of $H$.
If $\delta(\Gamma)\ge h/2$, then
\begin{align}
    \lambda_0(\Gamma\backslash H)=\delta(\Gamma)(h-\delta(\Gamma)).
\end{align}
Furthermore, $\Gamma\backslash H$ is $\lambda_0$-recurrent if $\Gamma$ is of divergence type.
These results are due to Sullivan in the real hyperbolic case \cite[Corollary 2.18]{Su3}.
Mutatis mutandis, his proof generalizes to cover all four types of hyperbolic spaces $H=H_\F^k$.

\begin{coro}\label{main3}
Let $H=H_\F^k$ and $\Gamma$ be of divergence type with $\delta(\Gamma)\ge h/2$.
Then $\Gamma$ admits an irreducible symmetric random walk with Martin boundary $\Lambda_\Gamma$.
\end{coro}

\cref{main3} provides a plethora of examples of Martin boundaries for random walks on groups.
It extends the main geometric result of \cite{BL2}, namely the corollary in the introduction there which is concerned with non-compact hyperbolic surfaces of finite area and their fundamental groups, which are finitely generated free groups.

Recall that for $H=H_\F^k$, the critical exponent $\delta(\Gamma)=\Hdim\Lambda_\Gamma^{\rad}$, the Hausdorff dimension of the \emph{radial limit set} $\Lambda_\Gamma^{\rad}$ of $\Gamma$ (this follows from bounds on the Patterson-Sullivan measures of shadows of balls \cite[Theorem 4.2]{Sch} and an approximation argument \cite[Proposition 4.10]{SV},  \cite[Theorem 0.3]{Nan}).
Moreover, if $\Gamma$ is geometrically finite (in the sense of Bowditch \cite{Bo}), then $\Lambda_\Gamma=\Lambda_\Gamma^{\rad}\cup P$, where $P$ is the set of parabolic fixed points of $\Gamma$ \cite[Theorem 2.3]{CI}.
Since $P$ is countable, this implies that
\begin{align}\label{hauscri}
    \delta(\Gamma)=\Hdim\Lambda_\Gamma^{\rad}=\Hdim\Lambda_\Gamma.
\end{align}
Furthermore, $\Gamma$ is of divergence type \cite[Proposition 3.7]{CI}.
The following, except for the last assertion, is now an immediate consequence of the above.

\begin{coro}\label{main3f}
Let $H=H_\F^k$ and $\Gamma$ be geometrically finite with $\Hdim\Lambda_\Gamma\ge h/2$.
Then $\Gamma$ admits an irreducible symmetric random walk with Martin boundary $\Lambda_\Gamma$.
Furthermore, if $\Hdim\Lambda_\Gamma>h/2$, then the random walk can be chosen of finite first moment and finite entropy.
\end{coro}

For the last assertion, we need that positive $\lambda_0$-harmonic functions on $M=\Gamma\backslash H$ are square-integrable.
Now \cite[Theorem B]{BP2} asserts that the bottom of the essential spectrum of $M$ is at least $\lambda_0(H)=h^2/4$, hence it is larger than $\lambda_0(M)=\delta(\Gamma)(h-\delta(\Gamma))$,
where we use that $\delta(\Gamma)=\Hdim\Lambda_\Gamma>h/2$.
Hence $\lambda_0(M)$ belongs to the discrete spectrum of $M$, and hence $\lambda_0$-harmonic functions are square-integrable.

Assume now again that $H$ is a Hadamard manifold with pinched negative curvature.

\begin{theo}\label{main5}
Suppose that $\Gamma$ is geometrically finite on $H$ and that $\xi\in\partial_{\geo}H$ is a parabolic fixed point of $\Gamma$.
Let $X=\Gamma x$ for some $x\in H$ and $\mu$ be an FLS-restriction of Sullivan's process $\mathcal S$ to $X$.
Then $k(.,\xi)=K(.,\xi)|_X$ is not $\mu$-harmonic.
\end{theo}

\begin{prob}
Characterize the points $\xi\in\Lambda_\Gamma$ for which $k(.,\xi)$ is $\mu$-harmonic.
\end{prob}

If $\Gamma$ is geometrically finite, then the complement $\Lambda_\Gamma^{\rad}$ of the parabolic fixed points in $\Lambda_\Gamma$ may be the set for which this is true.
At this point, we are able to handle the following special case.

\begin{theo}\label{main7}
Suppose that $\Gamma \backslash H$ has finite volume. Let $X$ be an orbit of $\Gamma$ and $\mu$ the LS-discretization of Brownian motion to $X$, associated to appropriate LS-data.
Then, for any $\xi \in \partial_{\geo}H$, we have that $k(.,\xi)$  is $\mu$-harmonic if and only if $\xi$ is not a parabolic fixed point of $\Gamma$.
\end{theo}

Our approach includes an application to a large class of possibly infinitely generated groups,
where the underlying spaces are Gromov hyperbolic spaces, not necessarily manifolds or even geodesic.
One of our main results in this direction is concerned with relatively hyperbolic groups (after Gromov \cite[Section 8.6]{Gro}, where we follow the exposition of Bowditch \cite{Bo1}).

\begin{theo}\label{bowbou}
Let $\Gamma$ be a relatively hyperbolic group  with finitely generated virtually nilpotent parabolic subgroups.
Let $\Gamma'\trianglelefteq \Gamma$ be a normal subgroup such that $\Gamma'\cap P$ has finite index in $P$, for any maximal parabolic subgroup $P$ of $\Gamma$.
If $\Gamma/\Gamma'$ is either finite, or virtually $\Z$, or virtually $\Z^2$, 
then there is an irreducible symmetric random walk on $\Gamma'$ with Martin boundary the Bowditch boundary of $\Gamma$.
\end{theo}

For example, let $\F_2=\la a,b\ra$ be the free group with two generators and $G=[\mathbb{F}_2,\mathbb{F}_2]$ its commutator subgroup, an infinitely generated free group.
Now $\F_2/G\cong\Z^2$, and hence we get a random walk $\mu$ in $G$ such that $\partial_\mu G=\partial_{\geo}\mathbb{F}_2$, the Gromov boundary of $\mathbb{F}_2$ with respect to the word metric.
However, if one considers $\F_2$ with the relatively hyperbolic structure $(\F_2,\la[a,b]\ra)$ corresponding to the action of $\F_2$ on $H_{\R}^2$ with quotient a hyperbolic surface of genus one with one cusp, then one obtains a different random walk $\mu'$ on $G$ such that $\partial_{\mu'}G=\partial_{\geo}(\F_2,\la[a,b]\ra)$, the Bowditch boundary, which in this case is a circle.

\cref{bowbou} is proved in \cref{subrelhyp}, where we also provide some definitions and details.
The case $\Gamma'=\Gamma$ in \cref{bowbou} is contained in \cite{Nan}.

In the text, we discuss versions of Theorems \ref{main1} and \ref{main5} for Ancona processes--in the sense of Definitions \ref{Proc} and \ref{Procp}--on Gromov hyperbolic spaces (not necessarily geodesic); cf.\,Theorems \ref{Martin} and \ref{nonharm}. 
The proofs of these two theorems are much more elementary in the case of the hyperbolic spaces $H_\F^k$.
The main reason for this is that then Martin kernels $K(.,\xi)$ are constant along horospheres centered at $\xi$,
for all $\xi$ in $\partial_{\geo}H$.
We indicate the corresponding easier proofs, using bisectors.
The latter are fundamental objects in geometry, which we now discuss shortly.

For two points $x,y$ in a metric space, their \emph{bisector} is the set of points in the space of equal distance to $x$ and $y$.
Here we need a geometric description of bisectors of pairs of different points in hyperbolic spaces.
In real hyperbolic geometry, the bisector is the hyperbolic hyperplane through the midpoint between the two points
and perpendicular to the geodesic segment connecting them.
In complex hyperbolic geometry, the situation is more complex, and bisectors are described in Goldman's \cite[Section 5]{Go}.
There is a corresponding description for quaternionic hyperbolic spaces and the hyperbolic octonionic plane,
and this might be known to experts.
However, we were not able to locate a reference for the latter two cases and discuss bisectors in an appendix.
The analog of Goldman's description is as follows.

\begin{propo}\label{bise}
Let $H=H_\F^k$ with $\F\in\{\C,\H,\O\}$ and $x\ne y$ be points in $H$.
Let $L$ be the unique totally geodesic real hyperbolic space in $H$ of dimension $\dim_\R\F$ containing $x$ and $y$ of curvature $-4$,
and let $z$ be a point in $H\setminus L$.
Then there is a unique totally geodesic real hyperbolic plane in $H$ of curvature $-1$ containing $x$, $z$, and $\pi z$, where $\pi\colon H\to L$ is the nearest point projection.
In particular,
\begin{align*}
    \cosh(d(x,z)) = \cosh(d(x,\pi z))\cosh(d(z,\pi z)).
\end{align*}
\end{propo}

The above $L$ is also called an \emph{$\F$-line}.
It is \emph{spanned} by $x\ne y$.
If, in the real hyperbolic case, the geodesic through $x$ and $y$ is understood to be the $\R$-line spanned by $x$ and $y$,
the following consequence holds in all four hyperbolic geometries.

\begin{coro}\label{bise2}
The bisector between $x\ne y$ in $H$ is equal to the preimage, under the nearest point projection onto the $\F$-line $L$ spanned by $x$ and $y$,
of the bisector between $x$ and $y$ in $L$.
\end{coro}

Recall that totally geodesic submanifolds of $H=H_\F^k$ are of codimension at least $d$ so that bisectors are \emph{not} totally geodesic unless $\F=\R$.
This is one of the reasons why the geometry of $H$ is more difficult for $\F\in\{\C,\H,\O\}$.

\subsection{Structure of the article}\label{struc}
As we already mentioned above, much of the discussion is in the setting of Gromov hyperbolic spaces.
These are introduced in \cref{sechyp}.
In \cref{secmark}, we fix the terminology regarding Markov processes and add some details about Martin boundaries. 
In \cref{secls}, we describe a refinement of the Lyons-Sullivan discretization procedure, which yields FLS-restrictions.
In \cref{secancona}, we introduce Ancona processes and prove general versions of \cref{main1} and Corollaries \ref{main3} and \ref{main3f}, where the proof of finiteness of first moments and entropy is transferred to \cref{subm}.
In \cref{suppc}, we identify the Poisson boundary of a general class of Ancona processes.
Sections \ref{subrelhyp} and \ref{secsem} are devoted to \cref{bowbou} and an extended version of \cref{main5}, respectively.
The proof of \cref{main7} is in \cref{secharm}.
The short Appendices \ref{symm} and \ref{bisec} are devoted to  hyperbolic planes in symmetric spaces and bisectors in hyperbolic spaces $H_\F^k$.

\subsection{Conventions}\label{convent}
Throughout the paper, we assume that $(H,\rho)$ is a proper metric space.
Usually, we just write $H$ instead of $(H,\rho)$, assuming that $\rho$ is given.
For $x\in H$ and $r>0$, $B(x,r)$ denotes the open ball of radius $r$ about $x$ and $B_x$ the ball $B(x,1)$.

Whenever occurring, $\mathcal{L}$ denotes a Radon measure on $H$, which is positive on open subsets of $H$.
Furthermore, $\Gamma$ denotes a group which acts properly discontinuously on $H$ and leaves given structures on $H$, like $\rho$ or $\mathcal L$, invariant.

\section{Gromov hyperbolic spaces}\label{sechyp}
Let $(H,\rho)$ be a proper metric space.
A (unit speed) \emph{geodesic} in $H$ is a curve $c\colon I\to H$ such that $\rho(c(s),c(t))=|s-t|$ for all $s,t\in I$.
We say that $H$ is \emph{geodesic} if, for any two points $x,y\in H$, there is a geodesic segment joining them.
Albeit not unique, we usually write $[x,y]$ for any such geodesic.

A (geodesic) triangle in $H$ consists of three geodesic segments which meet at the vertices $x,y,z$.
For simplicity, we denote such a triangle by $(x,y,z)$, albeit the geodesic segments $[x,y],[y,z],[z,x]$ connecting them might not be unique.

\subsection{Gromov hyperbolicity}\label{subgrohyp}
Given three points $x,y,z\in H$ (where $H$ is not necessarily geodesic), we call
\begin{align}\label{grpr}
    (yz)_x = \frac12(\rho(x,y) + \rho(x,z) - \rho(y,z)) 
\end{align}
the \emph{Gromov product} of $y,z$ with respect to $x$.
The three numbers
\begin{align*}
	a=(yz)_x, \quad b=(zx)_y, \quad c=(xy)_z \ge 0
\end{align*}
satisfy
\begin{align*}
    a+b=\rho(x,y), \quad a+c = \rho(x,z), \quad b+c = \rho(y,z).
\end{align*}
Given $u\in H$ and $\delta\ge0$, we say that $H$ is \emph{$\delta$-hyperbolic with respect to $u$} if
\begin{align}\label{grohyp}
    (xy)_u \ge (xz)_u\wedge(yz)_u - \delta     
\end{align}
for any three points $x,y,z$ in $H$.
If this holds for some $u$ in $H$, then $H$ is $2\delta$-hyperbolic with respect to any $u\in H$.
We say that $H$ is \emph{$\delta$-hyperbolic} if $H$ is \emph{$\delta$-hyperbolic with respect to any $u\in H$}.
We say that $H$ is \emph{Gromov-hyperbolic} if $H$ is \emph{$\delta$-hyperbolic} for some $\delta\ge0$.
General references for Gromov-hyperbolic spaces are \cite{Gro,GH,BriHae}.

Basic examples are complete and simply connected Riemannian manifolds with negative curvature bounded away from zero.
More precisely, for some appropriate $\delta>0$, such a manifold $H$ is $\delta$-hyperbolic if its  curvature is bounded from above by $-1$.
Simplicial trees with edges of positive lengths are $0$-hyperbolic. 

For a triangle $(x,y,z)$ in $H$, there is an isometric inclusion $i_\Delta$ from the triangle (as the union of its given edges) to the metric tripod $T_\Delta=T(a,b,c)$ with its three edges of respective lengths $a,b,c$ (as above).
We say that $(x,y,z)$ is \emph{$\delta$-thin} if preimages under $i_\Delta$ have diameter at most $\delta$.
The diameter of the preimage of the central vertex of $T_\Delta$ is called the \emph{insize} of $\Delta$.

For any triangle $\Delta=(x,y,z)$, we have
\begin{align}\label{triest}
	(yz)_x \le \rho(x,[y,z]) \le (yz)_x+\insize(\Delta).
\end{align}
If $H$ is geodesic and $H$ is $\delta$-hyperbolic, then all triangles in $H$ are $4\delta$-thin and \eqref{triest} implies
\begin{align}
	(yz)_x \le \rho(x,[y,z]) \le (yz)_x+4\delta.
\end{align}
Conversely, if all triangles in $H$ are $\delta$-thin, then $H$ is $2\delta$-hyperbolic.

We say that a triangle $(x,y,z)$ in $H$ is \emph{$\delta$-slim} if any edge (of its given edges) is contained in the $\delta$-neighborhood of the remaining two edges.
Obviously, if $(x,y,z)$ is $\delta$-thin, then it is $\delta$-slim with insize at most $\delta$.
Conversely, if $H$ is geodesic and all triangles in $H$ are $\delta$-slim, then all triangles in $H$ are $4\delta$-thin and $H$ is $8\delta$-hyperbolic.
See Proposition 21 \cite[Chapter 2]{GH} for all of this.

If $H$ is Gromov hyperbolic, then it admits a compactification by a boundary at infinity, called the \emph{Gromov} or \emph{geometric boundary} of $H$; see \cite[Chapter 7]{GH}.
To that end, say that a sequence $(x_m)$ in $H$ \emph{converges at infinity} if 
\begin{align*}
	\lim_{m,n\to\infty}(x_mx_n)_u=\infty
\end{align*}
for some--or any--$u\in H$.
Say that two such sequences $(x_m)$ and $(y_n)$ in $H$ are \emph{equivalent} if
\begin{align*}
	\lim_{m,n\to\infty}(x_my_n)_u=\infty
\end{align*}
for some--or any--$u\in H$.
The set $\partial_{\geo}H$ of equivalence classes is called the \emph{geometric boundary} of $H$.
If $\xi\in\partial_{\geo}H$ and the sequence $(x_m)$ belongs to $\xi$, then we say that $(x_m)$ \emph{converges to} $\xi$ and write $\lim x_m=\xi$.

The Gromov product extends to $\bar H=H\cup\partial_{\geo}H$:
For $u\in H$ and $\xi,\eta\in\bar H$, set
\begin{align*}
	(\xi\eta)_u = \sup\liminf_{m,n\to\infty}(x_my_n)_u
\end{align*}
where the supremum is taken over all sequences $(x_m)$ and $(y_n)$ converging to $\xi$ and $\eta$, respectively.
If $H$ is $\delta$-hyperbolic, then
\begin{enumerate}[label=(P\arabic*)]
\item\label{grpri} $\xi=\eta$ iff $2(\xi\eta)_{u}=\rho(u,\xi)+\rho(u,\eta)$ ($=\infty$ if $\xi$ or $\eta$ in $\partial_{\geo}H$).
\item\label{grpre} For all sequences $(x_m)$ and $(y_n)$ converging to $\xi$ and $\eta$,
\[(\xi\eta)_{u}-2\delta\le\liminf(x_my_n)_{u}\le(\xi\eta)_{u}.\]
\end{enumerate}

\begin{lem}\label{R2c}
Let $b$ be a Busemann function centered at a point $\xi\in\partial_{\geo}H$.
Let $x,y\in H$ with $b(y)-b(x)\le R$ and $(y\xi)_x\le c$.
Then $\rho(x,y)\le R+2c$.
\end{lem}

\begin{proof}
Given $\ve>0$, there is a point $z\in H$ close to $\xi$ such that
\begin{align*}
    \rho(y,x)+\rho(z,x)-\rho(y,z) &= 2(yz)_x \le 2(y\xi)_x + \ve \le 2c + \ve \\
    \rho(z,y)-\rho(z,x) &\le b(y)-b(x) + \ve \le R + \ve.
\end{align*}
Taking the sum of the equations and letting $\ve\to0$ yields the claim.
\end{proof}

We now define neighborhoods of points in $\bar H$.

\begin{dfn}\label{deftas}
Suppose that $H$ is $\delta$-hyperbolic, and let $x\in H$, $\xi\in\bar H$ and $r>0$.
Then we call
\begin{align*}
    R(x,\xi,r) := \{y\in \bar H \mid (y\xi)_x > r \}
\end{align*}
the \emph{region of width $1/r$ about $\xi$ with base $x$}.
\end{dfn}

Regions respectively their intersections $R(x,\xi,r)\cap\partial_{\geo}H$ turn $\bar H$ and $\partial_{\geo}H$ into compact Hausdorff spaces.

\begin{cor}\label{triest2}
Let $u\in H$ and $\xi,\eta\in\partial_{\geo}H$ with $\xi\ne\eta$.
Then $(xy)_u\leq(\xi\eta)_u+2\delta$ for all $x\in R(u,\xi,(\xi\eta)_u+\delta)$ and $y\in R(u,\eta,(\xi\eta)_u+2\delta)$.
\end{cor}

\begin{proof}
Since $H$ is $\delta$-hyperbolic,
\begin{align*}
    (\xi\eta)_u &\ge (\xi x)_u\wedge(x\eta)_u - \delta
    = (x\eta)_u - \delta \\         
    &\ge (xy)_u\wedge(y\eta)_u - 2\delta
    = (xy)_u -2\delta,
\end{align*}
where we used that $(\xi x)_u>(\xi\eta)_u+\delta$ and $(y\eta)_u>(\xi\eta)_u+2\delta$
for all $x\in R(u,\xi,(\xi\eta)_u+\delta)$ and $y\in R(u,\eta,(\xi\eta)_u+2\delta)$.
\end{proof}

Suppose now that $H$ is geodesic and $\delta$-hyperbolic in the sense of thin triangles.
Then, as in the case of Hadamard manifolds, points of $\partial_{\geo}H$ may be viewed as asymptote classes of geodesic rays.
For any two different points $\xi,\eta\in\partial_{\geo}H$, there is a geodesic $[\xi,\eta]$ joining them.
(The parametrization of geodesics will be specified if needed.)
By approximation, one concludes that geodesic triangles with vertices in $\bar H=H\cup \partial_{\geo}H$ are also $\delta$-thin.

The following estimate of distances between asymptotic rays is useful.

\begin{lem}\label{asym}
Let $[x,\xi)$ and $[y,\xi)$ be asymptotic geodesic rays, and let $x_t\in[x,\xi)$ and $y_t\in[y,\xi)$ be the points of distance $t$ on them to $x$ and $y$, respectively.
Then $\rho(x_t,[y,\xi))\le2\delta$ and $\rho(x_t,y_{b(y)+t})\le4\delta$ for all $t\ge \rho(x,y)+\delta$, where $b$ is the Busemann function associated to $[x,\xi)$.
\end{lem}

\begin{proof}
Let $t\ge \rho(x,y)+\delta$ and $s>t$.
Then there is a point $z\in[y,x_s]$ such that $\rho(x_t,z)\le\delta$.
By the triangle inequality, \[ |\rho(x_s,z)-(s-t)| = |\rho(x_s,z)-\rho(x_s,x_t)| \le \delta.\]
Therefore \[ \rho(y,z) = \rho(y,x_s) - \rho(z,x_s) = \rho(y,x_s) - (s-t) \pm \delta. \]
For $s$ sufficiently large, $\rho(x_s,z)\ge \rho(x_s,[y,\xi)) + \delta$, and then there is a point $y'\in[y,\xi)$ such that $\rho(y' ,z)\le\delta$.
This implies the first asserted inequality.
Furthermore, by the triangle inequality, $|\rho(y,y')-\rho(y,z)|\le\delta$.
Hence \[ \rho(y,y') =  \rho(y,x_s) - s + t \pm 2\delta.\]
Now, for any $\ve>0$, $\rho(y,x_s) - s = b(y) \pm \ve$ for all sufficiently large $s$.
Then $\rho(y_{b(y)+t},y')=\pm 2\delta\pm\ve$.
Letting $s$ tend to infinity, the second asserted inequality follows.
\end{proof}

\cref{4d2c} is a sophisticated version of \cref{asym} in the context of general Gromov hyperbolic spaces.

\subsection{Limit sets}\label{subseclim}
Let $H$ be a Gromov hyperbolic space.
For a subset $X\subseteq H$, let $\Lambda_X$ be the set of limit points of $X$ in $\partial_{\geo}H$.
If $\Gamma$ is a group of isometries of $H$, then $\Lambda_{\Gamma x}$ does not depend on the choice of $x$, and we write $\Lambda_\Gamma$ for $\Lambda_{\Gamma x}$ and call $\Lambda_\Gamma$ the \emph{limit set} of $\Gamma$.

Say that a diverging sequence $(x_m)$ in $X$ \emph{converges radially to} $\xi\in\partial_{\geo}H$ if $\limsup_{m\to\infty}(u\xi)_{x_m}<\infty$ for some (or any) $u\in H$.
Say that $\xi\in\partial_{\geo}H$ is a \emph{radial limit point of $X$} if there is a sequence in $X$ converging radially to $\xi$.

Denote by $\Lambda_X^{\rad}$ the set of radial limit points of $X$ in $\partial_{\geo}H$.
If $\Gamma$ is a group of isometries of $H$, then $\Lambda_{\Gamma x}^{\rad}$ does not depend on the choice of $x$, and we write $\Lambda_\Gamma^{\rad}$ for $\Lambda_{\Gamma x}^{\rad}$ and call $\Lambda_\Gamma^{\rad}$ the \emph{radial limit set} of $\Gamma$.

\begin{exa}\label{exaradial}
If there is a unit speed geodesic ray $c$ from $u$ to $\xi\in\partial_{\geo}H$, then $(u\xi)_{c(t)}=0$.
In particular, if $H$ is geodesic, then $\Lambda_H^{\rad}=\partial_{\geo}H$.
\end{exa}

\section{Markov processes}\label{secmark}

For $H$ and $\mathcal{L}$ as usual, choose an origin $o\in H$.
Let \[\mathcal{Z}=((Z_t)_{t\in \K},(\mathbb{P}_x)_{x\in H}),\] for $\K$ equal to $\N_0=\N\cup\{0\}$ or $[0,\infty)$, be a Markov process in $H$, where $H$ is assumed to be discrete if $K=\N_0$ and path connected if $\K=[0,\infty)$.
In the case $\K=\N_0$, we also speak of \emph{Markov chains}, in the case $\K=[0,\infty)$ of \emph{continuous Markov processes}.

\begin{enumerate}[label=(MP)]
    \item\label{conmar} Unless specified otherwise, we assume that Markov processes are \emph{irreducible}, \emph{time-homogeneous} and \emph{strong}.
    We also assume that sample paths of continuous Markov processes are $\p_x$-\emph{almost surely continuous} for any $x\in H$ and that they are \emph{absolutely continuous} in the sense defined below.
\end{enumerate}
For us, important examples of such processes will be time-homogeneous Markov chains on countable discrete sets and  diffusion processes \[dZ_t=b(Z_t)\,dt+\sigma\,dB_t\] on complete and connected Riemannian manifolds, where $(B_t)_{t\geq 0}$ is Brownian motion, $\sigma$ a diffusion coefficient, and $b$ a drift field.

For any non-empty open subset $V\subseteq H$, the \emph{Green kernel $G_V$ relative to $V$} is the family of Borel measures on $H$ defined by
\begin{align*}
    G_V(x,A):=\int_{\K} \p_x[Z_t\in A, t\leq  \tau_{V^c}]\,dt
\end{align*}
for points $x\in V$ and Borel subsets $A\subseteq H$,
where $\tau_{V^c}$ is the first hitting time for $V^c$ and $dt$ denotes the counting or Lebesgue measure, depending on $\K$.
We drop the subscript $V$ when $V=H$.
Clearly, $G_V\le G$ for any non-empty open $V\subseteq H$.

We say that $\mathcal Z$ is \emph{recurrent} if and only if, for any $x\in H$ and Borel set $A\subseteq H$, the probability, that a random path from $x$ meets $A$, is one.
If $\mathcal Z$ is not recurrent, then we say that it is \emph{transient}.
Equivalently, $\mathcal Z$ is transient if $G(x,.)$ is a Radon measure for one (or any) $x\in H$.

We say that a continuous Markov process $\mathcal{Z}$ on $H$ is \emph{absolutely continuous} if, for any $x_0\in H$, $0<2r<\diam H$, and $x\in B(x_0,r)=V$, $\tau_{V^c}$ is $\p_x$-almost surely finite and $G_V(x,.)$, is a finite Radon measure on $V$, which is absolutely continuous with respect to $\mathcal L$ on $V$ with density $G_V=G_V(x,y)$ on $V^2$, which is  continuous outside the diagonal.
In the transient case, we also require that $G$ satisfies the corresponding properties.

\subsection{Martin boundary and hitting measures}\label{secmart}
Assume that $\mathcal{Z}$ is a transient Markov process on $H$ (as in \ref{conmar}).
In the continuous case, assume also that the associated Martin kernels $K(.,y)=G(.,y)/G(o,y)$ are locally equicontinuous on $H$ in the sense that, for any compact subset $C\subseteq H$, the set of $K(.,y)$ with $\rho(y,C)\ge1$ is equicontinuous on $C$.

Consider the space of all divergent sequences $(y_n)$ in $H$.
By the above, the associated Martin kernels $K(.,y_n)=G(.,y_n)/G(o,y_n)$ are locally equicontinuous on $H$.
Consider diverging sequences $(y_n)$ such that the associated Martin kernels converge to a limit function, and define two such sequences to be equivalent if their limit functions coincide.
The set $\partial_{\mathcal{Z}}H$ of equivalence classes is called the \emph{Martin boundary} of $(H,\mathcal{Z})$.   
Identifying points $x\in H$ with the corresponding ($\mathcal{Z}$-)super-harmonic functions $K_x$ embeds $H$ into the space $\hat{\mathcal S}(H,\mathcal Z)$ of $\mathcal Z$-super-harmonic functions on $H$.
With respect to this identification, uniform convergence on compact subsets turns $H\cup\partial_{\mathcal{Z}}H$ and $\partial_{\mathcal{Z}}H$ into compact Hausdorff spaces.
For convenience, we sometimes distinguish points $\xi$ in $\partial_{\mathcal{Z}}H$ from the super-harmonic functions $K_\xi=K(.,\xi)$ they represent.

We have the \emph{Riesz decomposition} of positive super-harmonic function as the unique sum of a non-negative potential and a non-negative ($\mathcal{Z}$-)harmonic function; see \cite{KunWata}, \cite{KunWatp}.
The set $\partial_{\mathcal{Z}}^{\min}H$ of extremal positive harmonic functions, called the \emph{minimal Martin boundary}, is a Borel subset of $\partial_{\mathcal{Z}}H$.
For any positive harmonic function $u$, we have a \emph{Martin representation formula}, that is, there is a unique finite Radon measure $\mu_u$ on $\partial_{\mathcal{Z}}^{\min}H$ such that
\begin{align*}
	u(x)=\int K(x,\xi)\,d\mu_u(\xi);
\end{align*}
see \cite[Theorem 4]{KunWatp}.
Moreover one has a \emph{Martin convergence theorem}, that is, there is a random variable $Z_\infty$ taking values in $\partial_{\mathcal{Z}}^{\min}H$ such that, for each $x\in H$, $Z_t$ converges $\p_x$-almost surely to $Z_\infty$;
a proof may be obtained as a slight variation of the argument in \cite[Theorem 5.1]{Saw}.
The probability measure $(Z_\infty)_\ast \p_x$ in $\partial_{\mathcal{Z}}^{\min}H$ is the \emph{hitting} or \emph{exit measure relative to $x$}, denoted $\nu_x$.
The Radon-Nikodym derivative of the hitting measures is
\begin{align*}
	\frac{d\nu_x}{d\nu_y}(\xi)=\frac{K(x,\xi)}{K(y,\xi)},
\end{align*}
for all $x,y\in H$ and $\xi\in\partial_{\mathcal{Z}}^{\min}H$ (see \cite[Eq. 12.4, Proposition 12.11]{KunWatp}). Consequently, 
\begin{align*}
	\nu_x(E)=\int_E K(x,\xi) \,d\nu_{o}(\xi),
\end{align*}
for any $x\in H$ and Borel subset $E\subseteq\partial_{\mathcal{Z}}^{\min}H$. 

\begin{exa}\label{SullMart}
Let $H$ be a Riemannian symmetric space with negative  curvature,
that is, $H$ is a hyperbolic space, $H=H_\F^k$ with $\F\in\{\R,\C,\H,\O\}$, endowed with its standard metric.
For a quotient $M=\Gamma\backslash H$ with $\Gamma$ of divergence type and of critical exponent $\delta=\delta(\Gamma)\ge h/2$,
let $\vf$ be the lift of the positive $\lambda_0$-harmonic function on $M$ to $H$ with $\vf(o)=1$, where $\lambda_0=\delta(h-\delta)$.
Then the Martin boundary of Sullivan's $\vf$-process on $H$ coincides with the limit set $\Lambda_\Gamma\subseteq\partial_{\geo}H$ at infinity.
Using that $\vf(o)=1$, the Martin kernels are given by
\begin{align*}
    K(x,\xi) = e^{-\delta b_\xi(x)}/\vf(x),
\end{align*}
where $\xi\in\Lambda_\Gamma$ and $b_\xi$ is the Busemann function associated to $\xi$ such that $b_\xi(o)=0$ for the chosen origin $o\in H$.
\end{exa}

\section{Lyons-Sullivan discretization}\label{secls}

With $(H,\rho,\mathcal{L})$ as usual, we assume that $H$ is path connected and let $\mathcal Z=(Z_t)_{t\ge0}$ be a continuous Markov process on $H$ (satisfying \ref{conmar}), invariant under a group $\Gamma$ of isometries acting properly discontinuously on $H$.
We say that $\mathcal{Z}$ is \emph{regular} if, for any open ball $V\subseteq H$ of radius less than $\diam H/2\le\infty$
\begin{enumerate}[label=(R\arabic*)]
\item\label{r1}
and compact subset $F\subseteq V$, there is a constant $C>1$ such that
\begin{align*}
    d\varepsilon_x/d\varepsilon_y = C^{\pm 1}
\end{align*}
for all $x,y\in F$, where $\ve_z=(Z_{\tau_{\partial V}})_{\ast}\p_z$;
\item\label{r2}
and $x\ne y$ in $V$, $G_V(x,y)\to\infty$ as $y\to x$.
\end{enumerate}
These properties are satisfied for Brownian motion and, more generally, for Sullivan's process on a (complete and connected) Riemannian manifold.

\subsection{Discretization}\label{ssecls}
Let $\mathcal{Z}=((Z_t)_{t\ge0},(\p_x)_{x\in H})$ be a Markov process in $H$ (in accordance with \ref{conmar}),
and $X$ a discrete subset of $H$.
Let $(F_x)_{x\in X}$ and $(V_x)_{x\in X}$ be families of compact respectively open and relatively compact subsets of $H$.
We say that they are (regular Lyons-Sullivan) \emph{$LS$-data for $X$ (with respect to $\mathcal{Z}$)} if
\begin{enumerate}[label=(D\arabic*)]
\item\label{d1}
$x\in\mathring{F}_x$ and $F_x\subseteq V_x$ for all $x\in X$;
\item\label{d2}
$F_x\cap V_y=\emptyset$ for all $x\ne y$ in $X$;
\item\label{d3}
$F=\cup_{x\in X}F_x$ is closed and recurrent with respect to $\mathcal Z$;
\item\label{d4}
for all $x\in X$ and $y\in F_x$, the exit measures from $V_x$ satisfy
\begin{equation*}
  d\ve_x/d\ve_y = C^{\pm1}
\end{equation*}
for some constant $C>1$, which does not depend on $x$ and $y$.
\end{enumerate}
Lyons and Sullivan associate to such data a family $\mu=(\mu_y)_{y\in H}$ of probability measures on $X$ such that restriction from $H$ to $X$ defines an isomorphism between the spaces of bounded $\mathcal{Z}$-harmonic functions and bounded $\mu$-harmonic functions on $X$.
In the situation they call cocompact, they obtain also an isomorphism of the corresponding cones of positive harmonic functions.
They actually discuss all this in the case where $H$ is a Riemannian manifold and $\mathcal{Z}$ is Brownian motion.
However, their arguments also hold in greater generality \cite{BP}. 

We will need some information on the construction of $\mu$.
To that end, suppose that we are given LS-data as above, and let $\Omega$ be the space of sample paths of $\mathcal{Z}$.
For $\omega\in\Omega$, set
\begin{equation}\label{s0}
  S_0(\omega)= 
  \begin{cases}
  0 &\text{if $\omega(0)\notin X$}, \\
  \inf \{ t>0 \mid \omega(t)\notin V_x\} &\text{if $\omega(0)=x\in X$,}
\end{cases}
\end{equation}
and recursively, for $n\ge1$,
\begin{equation}\label{rs}
\begin{split}
  R_n(\omega) &= \inf \{ t\ge S_{n-1}(\omega) \mid \omega(t)\in F \}, \\
  S_n(\omega) &= \inf \{ t\ge R_n(\omega) \mid \omega(t)\notin V_{X_n(\omega)} \},
\end{split}
\end{equation}
where $X_n=X_n(\omega)\in X$ with $Y_n=Y_n(\omega)=\omega(R_n(\omega))\in F_{X_n(\omega)}$.

On the product $\tilde\Omega=\Omega\times[0,1]^\N$, let
\begin{equation}\label{na}
	N(\omega,\alpha)
	= \inf \{ n>0 \mid \alpha_n < \kappa(X_n(\omega),Y_n(\omega),Z_n(\omega)) \}
\end{equation}
where $\alpha=(\alpha_1,\alpha_2,\dots)\in[0,1]^\N$, $Z_n=Z_n(\omega)=\omega(S_n(\omega))$, and
\begin{align}\label{nk}
	\kappa = \kappa(x,y,z)
	= \frac1C\frac{d\ve(x,V_x)}{d\ve(y,V_x)}(z)
\end{align}
for $x\in X$, $y\in F_x$ and $z\in\partial V_x$. With $\hat{\p}_y=\p_y\otimes \lambda^{\otimes\N}$, where $\lambda$ is the Lebesgue measure in $[0,1]$, we have
\begin{align}\label{muy}
    \mu_y(x)=\hat{\p}_y[Z_\tau\in F_x],
\end{align}
for all $y\in H$ and $x\in X$, where the \textit{stopping time $\tau$ associated to $\mu$} is
\begin{align}\label{tau}
	\tau(\omega,\alpha)=S_{N(\omega,\alpha)}(\omega);
\end{align}
see \cite[Chapter 8]{LS}.

\begin{rem}\label{s1}
If $X$ and $\mathcal{Z}$ are invariant under $\Gamma$ and the quotient process of $\mathcal Z$ in $\Gamma\backslash H$ has an invariant probability, then, by Kac's return time lemma \cite{Ka}, applied to the quotient of the Markov chain $(Z_i)_{i=1}^\infty$, there is a constant $T>0$ such that $E_y[S_1] < T$ for all $y\in \cup_{x\in X}\partial V_x$ ; compare with the proof of \cite[Theorem 3.2(b)]{BL2}, where the case of Brownian motion is considered.
\end{rem}

\begin{rem}\label{poisson}
If $\mu$ is the LS-discretization of a regular process $\mathcal Z$, then the Poisson boundaries and the spaces of bounded harmonic functions of $\mathcal Z$ and $\mu$ are isomorphic. If $\mu$ is supported in a $\Gamma$-invariant discrete set and $\mathcal Z$ is $\Gamma$-invariant, then the isomorphism is $\Gamma$-equivariant. For proofs, note that the arguments of \cite{Kai}, \cite{BL} adapt directly to the more general setting.
\end{rem}

We say that the $LS$-data are \emph{balanced} if there is a constant $B>0$ such that the Green kernels of the $V_x$ satisfy
\begin{enumerate}[label=(D5)]
\item\label{balan}
$G_{V_x}(z,x) = B$ for all $x\in X$ and $z\in\partial F_x$.
\end{enumerate}
In the case of Brownian motion on Riemannian manifolds, the notion of balanced was introduced in \cite{BL2}; see also \cite[(D5) in Section 3]{BP}.
It may be checked by following the proofs in \cite{LS} and \cite{BL2} verbatim that the results on the existence of (balanced) LS-data may be extended to processes which are regular in the sense defined in the beginning of this section.

Let $\Gamma$ be a discrete subgroup of the isometry group of $H$. We say that $\mathcal Z$ is \textit{recurrent modulo $\Gamma$}, if $\mathcal Z$ is $\Gamma$-invariant, and almost every $\mathcal Z$-path modulo $\Gamma$ visits any open set in $\Gamma\backslash H$ infinitely often. 

\begin{prop}\label{propad}
For a properly discontinuous group $\Gamma$ of automorphisms of $(H,\mathcal{Z})$,
an orbit $X\subseteq H$ admits $\Gamma$-invariant balanced $LS$-data if $\mathcal{Z}$ is regular and recurrent modulo $\Gamma$.
\end{prop}

\begin{proof}
Let $X=\Gamma o$ and $\Gamma_{o}$ be the stabilizer of $o$.
Let $V_{o}$ be an open ball centered at $o$ (invariant under $\Gamma_{o}$). Define $F_{o}:=\{z\in V_{o}\mid G_{V_{o}}(z,o) \geq t\}$, for $t>0$ large enough that $F_{o}$ is compactly contained in $V_{o}$, and such that $gF_{o}\cap V_{o}=\emptyset$ unless $g\in\Gamma_{o}$. 
For $g\in\Gamma$ and $x=go$, let $F_{x}=gF_{o}$ and $V_{x}=gV_{o}$.
Since $\mathcal{Z}$ is recurrent modulo $\Gamma$, $F=\cup_{x\in X}F_x$ is $\mathcal{Z}$-recurrent.
Moreover, $(F,V)=(F_x,V_x)_{x\in X}$ satisfies \ref{d4} by regularity of $\mathcal Z$. 
By the choice of $F_o$, \ref{balan} holds for $B = t$.
In conclusion, $(F,V)$ are $\Gamma$-invariant balanced LS-data.
\end{proof}

\begin{thm}\label{Green}
If $(F_x,V_x)$ are balanced $LS$-data for $X$ with constants $B$ and $C$ and associated family $\mu$ of probability measures on $X$, then
\begin{align*}
	G(y,x) = BCg(y,x)
	\quad\text{for all $x\in X$ and $y\in H\setminus V_x$},
\end{align*}
where $g$ denotes the Green kernel of the $\mu$-random walk on $X$.
In particular, 
\begin{enumerate}
    \item $\mathcal{Z}$ is transient if and only if the $\mu$-random walk is transient;
    \item the Martin kernels $K$ of $\mathcal{Z}$ and $k$ of $\mu$ are related by
    \begin{align*}
        k(y,x)=K(y,x)
    \end{align*}
    for all $y\in H\setminus (V_x \cup V_o)$ and $x\ne o$, where we choose $o\in X$.
\end{enumerate}
Moreover, if $g$ is symmetric, then $\mu_{x}(y) = \mu_{y}(x)$ for all $x,y \in X$.
\end{thm}

For the proof of \cref{Green}, compare \cite[p.\,86-88]{BL2}.
It establishes, in particular, that \emph{$\mu$ is an FLS-restriction} in the sense of the introduction.

\subsection{On the support of hitting measures of LS-discretizations}\label{subsupp}
Let $(H,\rho,\mathcal L)$ be as in the beginning of \cref{secls}.
Let $\mathcal{Z}=(Z_t)_{t\ge0}$ be a transient continuous Markov process in $H$ and $\nu=(\nu_x)_{x\in H}$ be the family of hitting measures in $\partial_{\mathcal Z}H$ associated to $\mathcal Z$.
Recall that the $\nu_x$ are pairwise equivalent with Radon-Nikodym derivatives $d\nu_x/d\nu_y=K(x,.)/K(y,.)$.
Hence the support of the $\nu_x$ does not depend on $x$; we denote it by $\supp\nu$.

\begin{prop}\label{supp}
Let $X\subseteq H$ be a discrete subset, which admits LS-data $(F,V)$ relative to a Markov process $\mathcal Z$.
Then $\supp \nu\subseteq \Lambda_F^{\mathcal Z}$, the set of limit points of $F$ in $\partial_{\mathcal Z}H$.
\end{prop}

\begin{proof}
If $\supp\nu$ would not be contained in $\Lambda_F^{\mathcal Z}$, then there would be a neighborhood $U\subseteq H\cup\partial_{\mathcal Z}H$ of a point $\xi\in \supp\nu \cap \partial_{\mathcal Z}H$ such that $\nu_o(U\cap \partial_{\mathcal Z}H)>0$ and $F\cap U=\emptyset$.
But since $\nu_o(U\cap \partial_{\mathcal Z}H)>0$, this would imply that, for a set of sample paths of positive measure, the final exit time from $F$ would be finite.
But then $F$ would not be recurrent with respect to $\mathcal{Z}$, a contradiction. 
\end{proof}

\section{Ancona processes}\label{secancona}
For $(H,\rho,\mathcal L)$ as usual, let $\mathcal{Z}=((Z_t)_{t\in \K},(\mathbb{P}_x)_{x\in H})$ be a Markov process in $H$ (in the sense of \ref{conmar}), $\Gamma$ a group acting properly discontinuously by automorphisms on $(H,\rho,\mathcal L,\mathcal{Z})$, and $o\in H$ an origin.
Assume that $H$ is $\delta$-hyperbolic and $\mathcal Z$ is transient with Green kernel $G$ and Martin kernel $K$ with respect to $o$. 

\begin{dfn}\label{Proc}
We say that $\mathcal{Z}$ is an \emph{Ancona process} if 
\begin{enumerate}[leftmargin=*, label=(A\arabic*)]
\item\label{decay}
there are $\alpha>0$ and $C_1\ge 1$ such that for all $x,y,z\in H$ satisfying $\rho(x,y) \geq 2$, we have
\begin{equation*}
    G(x,B_y)G(y,B_x)\le C_1 e^{-\alpha \rho(x,y)}
    \;\;\;\text{and}\;\;\; G(z,B_z) = C_1^{\pm1};
\end{equation*}
\item\label{harnackin}
(\emph{Harnack inequality})
for any $r,\ve>0$, there is $C_2=C_2(r,\ve)>0$ such that, for all $x,y,z\in H$ and $u,v,w\in H$ with 
\[\min\{\rho(x,y),\rho(x,z)\}>\ve\;\;\;\text{and}\;\;\; \max\{\rho(y,z),\rho(v,w)\}<r,\] 
\[\text{we have}\;\;\; G(x,y)\leq C_2 G(x,z)\;\;\;\text{and}\;\;\; G(v,B_u)\leq C_2 G(w,B_u);\]
\item\label{ancin}
(\emph{Ancona inequality})
for any $r\geq 0$, there is $C_3=C_3(r)>0$ such that, for all $x,y,z$ in $H$ with $\rho(x,y)\geq 2$ and $(xy)_z\le r$, we have 
\begin{equation*}
    G(x,B_y)\leq  C_3 G(x,B_z) G(z,B_y);
\end{equation*}
\item\label{geomar}there is a $\Gamma$-equivariant homeomorphism $\partial_{\geo}H\to\partial_{\mathcal Z}H$ which extends the identity of $H$ continuously.
\end{enumerate}
\end{dfn}

\begin{rem}
The above properties (as written, as well as with the balls of radius one above replaced by their centers in the first part of \ref{decay} and in \ref{ancin}) are satisfied by a large class of diffusion processes (see \cite{An}) and random walks (see \cite{An1, BHM, Gou}).
In particular, if $\mathcal L(B_y)=a^{\pm 1}$, for some $a>0$ and any $y\in H$, then by \ref{harnackin}, for any $x,y\in H$ with $\rho(x,y)\geq 2$, $G(x,B_y)=b^{\pm 1}G(x,y)$, for a constant $b>0$.  An application where the properties above are not however satisfied with the balls replaced by their respective centers appears in \cref{subrelhyp}.
\end{rem}

\begin{thm}\label{suppc}
For an Ancona process $\mathcal Z$ as above, assume in addition that $\mathcal{Z}$ is regular and recurrent modulo $\Gamma$, $H$ is geodesic, and $\Gamma$ does not have any fixed points in $\partial_{\geo}H$.
Then, under the identification $\partial_{\geo}H=\partial_{\mathcal Z}H$, the hitting measures $\nu=(\nu_x)_{x\in H}$ of $\mathcal{Z}$ on $\partial_{\geo}H$ have support $\supp\nu=\Lambda_\Gamma$.
\end{thm}

\begin{proof}
Since $\mathcal{Z}$ is recurrent modulo $\Gamma$, orbits of $\Gamma$ admit $\Gamma$-invariant LS-data $(F,V)$, by \cref{propad}.
Hence under the identification $\partial_{\geo}H=\partial_{\mathcal Z}H$, we have $\supp\nu\subseteq\Lambda_F$ by \cref{supp}.
Now by $\Gamma$-invariance, the diameters $\diam F_x$ do not depend on $x$.
Hence $\Lambda_\Gamma=\Lambda_F$ by hyperbolicity and, therefore, $\supp\nu\subseteq\Lambda_\Gamma$.
Since $\Gamma$ does not have fixed points in $\partial_{\geo}H$, the action of $\Gamma$ on $\Lambda_\Gamma$ is minimal; see Corollaire 26 in \cite[Chapter 8]{GH}.
Since $\supp\nu$ is non-empty, closed, and $\Gamma$-invariant, we get $\supp\nu=\Lambda_\Gamma$.
As for the claim about the Poisson boundary, see \cref{poisson}.
\end{proof}

If $H$ is a Hadamard manifold with pinched negative  curvature, then Brownian motion $\mathcal B$ is an Ancona process on $H$.
More generally, we have the following important class of examples.

\begin{thm}\label{sullivan}
Suppose that $M=\Gamma\backslash H$ is $\lambda_0$-recurrent, where $H$ is a Hada\-mard manifold with pinched negative  curvature, and that one of the following two conditions holds:
\begin{enumerate}
\item\label{less0} $\lambda_0=\lambda_0(M)<\lambda_0(H)$;
\item\label{less2} $\lambda_0=\lambda_0(M)=\lambda_0(H)$ and $H$ admits a compact quotient.
\end{enumerate}
Then Sullivan's $\vf$-process $\mathcal S$ is a $\Gamma$-invariant Ancona process on $H$.
\end{thm}

\begin{proof}
We start with proving \eqref{less0}.
Since $\lambda_0<\lambda_0(H)$, the operator $-\Delta+\lambda_0$ is coercive on $H$ in the sense of Ancona \cite[Section 1]{An}.
Hence, for this operator, \ref{decay} follows from Propositions 7 and 10 of \cite{An},
\ref{harnackin} is just the standard Harnack inequality for elliptic operators,
\ref{ancin} is a form of Theorem 1 of \cite{An}--the Harnack inequality at infinity--,
and \ref{geomar} is part of Theorem 3 of \cite{An}.

To obtain the corresponding properties for Sullivan's $\vf$-process $\mathcal S$, we note that the Green kernel of $\mathcal S$ is given by
\begin{align}
    G_\varphi(x,y) = G_{\lambda_0}(x,y)\frac{\vf(y)}{\vf(x)},
\end{align}
where $G_{\lambda_0}$ denotes the Green kernel of $-\Delta+\lambda_0$.
Now the standard Harnack inequality for elliptic operators in \ref{harnackin} remains true anyway and the factor $\vf(y)/\vf(x)$ just cancels out in the right way in the remaining properties.
Here we remark that the Martin kernel of $\mathcal S$ is
\begin{align}
    K(x,y) = \frac{G_{\lambda_0}(x,y)}{G_{\lambda_0}(o,y)}\frac{\vf(o)}{\vf(x)} = K_{\lambda_0}(x,y)\frac{\vf(o)}{\vf(x)}
\end{align}
so that convergence of a sequence of $K(x,y_n)$ is equivalent to convergence of the corresponding sequence of $K_{\lambda_0}(x,y_n)$.

Concerning \eqref{less2}, the arguments for the first part of the above proof follow from \cite{LL}:
\ref{decay} corresponds to \cite[Theorem 1.3]{LL}, \ref{ancin} to \cite[Theorem 3.2]{LL}, and \ref{geomar} to \cite[Theorem 1.4]{LL}.
The second part of the proof is exactly the same as above.
\end{proof}

Our first aim in this section is now \cref{dense}, for which we need some preparations.

For an Ancona process, by \eqref{harnackin}, there is $C_0>0$ such that, for any $x,y,z\in H$ with $B_x\cap B_z=\emptyset$ and $B_{y}\cap B_z=\emptyset$,
\begin{equation}\label{Mapprox}
    \frac{G(x,z)}{G(y,z)}= C_0^{\pm 1}\cdot\frac{G(x,B_z)}{G(y,B_z)}
\end{equation}
so that, for any sequence $z_n\xrightarrow{n\to\infty}\xi\in\partial_{\mathcal{Z}}H$,
\begin{align*}
  C_0^{-1}\limsup_{n\to\infty} \frac{G(x,B_{z_n})}{G(o,B_{z_n})}\leq K(x,\xi)\leq C_0\liminf_{n\to\infty} \frac{G(x,B_{z_n})}{G(o,B_{z_n})}.
\end{align*}

The following general consequence of the Harnack inequality of $\Gamma$-Ancona processes is an inverse of \ref{ancin}.

\begin{lem}\label{Gtriangle}
Suppose $\mathcal{Y}$ has the property \ref{conmar} in the metric space $(H,\rho)$, satisfies the second inequality of \ref{harnackin}, and $G(x,B_x)=C_1^{\pm 1}$ for some $C_1\geq 1$ and all $x\in H$.
Then there is $C'>0$ such that, for any $x,y,z\in H$,
\begin{equation*}
G(x,B_z)\geq  C'\cdot G(x,B_y)\cdot G(y,B_z).
\end{equation*}
\end{lem}

\begin{proof}
For any $w\in H$, let $\tau_w$ be the first hitting time for $B_w$.
Write
\begin{align}\label{F}
\begin{split}
    F(x,B_y) :&= \p_x[Y_t\in B_y \;\text{for some}\; t\geq 0] \\
    &= \p_x[\tau_y<\infty] = \int_{H} d[Y_{\tau_y*}\p_x](y').   
\end{split}
\end{align}
Then we have, with an appropriate constant $C>1$,
\begin{align}\label{Greeneq}
\begin{split}
    G(x,B_y) &= \int_0^\infty \p_x[Y_t\in B_y]\,dt
    =\int_{H} G(y',B_y)d[Y_{\tau_y*}\p_x](y') \\
    &\overset{(a)}{=} C^{\pm 1}\cdot F(x,B_y)\cdot G(y,B_y) = (CC_1)^{\pm 1} F(x,B_y)
\end{split}    
\end{align}
and
\begin{align}\label{triangle}
\begin{split}
F(x,B_z) &\geq \p_x[Y_t\in B_z\;\text{for some}\; t\geq \tau_y] \\
&= \int_{H} F(y',B_z)\, d[(Y_{\tau_y})_\ast\p_x](y') \\
&\overset{(b)}{=}C^{\pm 1}\cdot  F(x,B_y)\cdot F(y,B_z),
\end{split}
\end{align}
where (a) follows from the second inequality in \ref{harnackin} and (b) from (a) and the second inequality in \ref{harnackin}. So \eqref{Gtriangle} is obtained from \eqref{Greeneq} and \eqref{triangle}.
\end{proof}

\begin{lem}\label{lemnotan}
For $x_0\in H$ and $\xi\in\partial_{\geo}H$,  $(ox)_{x_0}\le (o\xi)_{x_0}+\delta$ for all $x\in R(x_0,\xi,(o\xi)_{x_0}+\delta)$.
\end{lem}

\begin{proof}
Since $H$ is $\delta$-hyperbolic, we have $(o\xi)_{x_0} \ge (ox)_{x_0}\wedge(x\xi)_{x_0} - \delta$ for all $x\in H$.
By definition, $(x\xi)_{x_0}>(o\xi)_{x_0}+\delta$ for $x\in R(x_0,\xi,(o\xi)_{x_0}+\delta)$.
This implies that $(ox)_{x_0}\wedge(x\xi)_{x_0}=(ox)_{x_0}$.
\end{proof}

From now on, we assume that $\mathcal Z$ is an Ancona process.

\begin{lem}\label{lemdense}
For $x_0\in H$ with $\rho(o,x_0)\geq 2$ and  $\xi\ne\eta$ in $\partial_{\geo}H$,
\begin{align*}
    K(x,\eta)/K(x,\xi) \le C e^{-\alpha\rho(o,x_0)},
\end{align*}
for all $x\in R(o,\xi,(\xi\eta)_o+\delta)\cap R(x_0,\xi,(o\xi)_{x_0}+\delta)$,
where $C$ is given in \eqref{Ceq}.
\end{lem}

\begin{proof}
Let $x\in U = R(o,\xi,(\xi\eta)_o+\delta)\cap R(x_0,\xi,(o\xi)_{x_0}+\delta)$ and $y\in V=R(o,\eta,(\xi\eta)_o+2\delta)$.
Then we get $(xy)_o\leq(\xi\eta)_o+2\delta$ for all $x\in U$ and $y\in V$, by \cref{triest2}.
In particular, \ref{ancin} applies to triples $x\in U$, $y\in V$ and $o$ with $C_3=C_3((\xi\eta)_o+2\delta)$.

Let $x\in U$ and $(y_k)$ be a sequence in $V$ converging to $\eta$.
Then  $K(x,y_k)\to K(x,\eta)$ as $k\to\infty$, by \ref{geomar}.
From \eqref{Mapprox} and \ref{ancin}, we obtain
\begin{align*}
    K(x,y_k) &= \frac{G(x,y_k)}{G(o,y_k)} \le C_0 \frac{G(x,B_{y_k})}{G(o,B_{y_k})} \\
    &\le C_0C_3 \frac{G(x,B_{o})G(o,B_{y_k})}{G(o,B_{y_k})} = C_0C_3G(x,B_{o}).
\end{align*}
Let now $(z_l)$ be a sequence in $U$ converging to $\xi$.
Then $K(x,z_l)\to K(x,\xi)$ as $l\to\infty$ and
\begin{align*}
    K(x,z_l) &= \frac{G(x,z_l)}{G(o,z_l)}
    \ge \frac1{C_0}\frac{G(x,B_{z_l})}{G(o,B_{z_l})} \\
    &\ge \frac{C'}{C_0C_4}\frac{G(x,B_{x_0})G(x_0,B_{z_l})}{G(o,B_{x_0})G(x_0,B_{z_l})}
    = \frac{C'}{C_0C_4}\frac{G(x,B_{x_0})}{G(o,B_{x_0})},
\end{align*}
by \eqref{Mapprox} and \ref{ancin} with $C_4=C_3((o\xi)_{x_0}+\delta)$, the latter by \cref{lemnotan}.
Hence
\begin{align}
    \frac{K(x,y_k)}{K(x,z_l)}
    &\le \frac{C_0^2C_3C_4}{C'}\frac{G(x,B_{o})G(o,B_{x_0})}{G(x,B_{x_0})} \notag \\
    &\le \frac{C_0^2C_3C_4^2}{C'}\frac{G(x,B_{x_0})G(x_0,B_o)G(o,B_{x_0})}{G(x,B_{x_0})} \notag \\
    &= \frac{C_0^2C_3C_4^2}{C'}G(x_0,B_o)G(o,B_{x_0}) \notag \\
    &\le  \frac{C_0^2C_1C_3C_4^2}{C'}e^{-\alpha\rho(o,x_0)}
    = C e^{-\alpha\rho(o,x_0)}, \label{Ceq}
\end{align}
where $C_1$ is the constant from \ref{decay}.
\end{proof}

The factor $C_4^2$ of the constant $C$ in \eqref{Ceq} depends on the location of $x_0$; more precisely, on $(o\xi)_{x_0}$.
To get a hand on this dependence, the notion of radial convergence comes into play; see \cref{subseclim}.

\begin{cor}\label{dense}
For any $\xi\ne\eta$ in $\partial_{\geo}H$ and set $X\subseteq H$, if $\xi\in\Lambda_H^{\rad}\cap\Lambda_X$, then $K(\cdot,\xi)|_X\neq K(\cdot,\eta)|_X$.
\end{cor}

\begin{proof}
Choose a sequence $(x_m)$ in $H$ which converges radially to $\xi$.
Then with $x_m$ in place of $x_0$, we can apply \cref{lemdense}, where the constant $C$ does not depend on $m$ and get that, for any $m$, $K(x,\xi)/K(x,\eta)$ is arbitrarily small for all $x\in X$ sufficiently close to $\xi$. (Not equal to one would be sufficient.) Such $x$ exist since $\xi\in\Lambda_X$.
\end{proof}

\begin{rem}\label{lemdenss}
In the case of Sullivan's process on hyperbolic spaces $H^k_\F$ and $X$ an orbit of $\Gamma$, there is a more elementary proof of \cref{dense}.
Namely, by \cref{SullMart},
\begin{align*}
    K(x,\eta)/K(x,\xi)  = e^{\delta(b_\xi(x)-b_\eta(x))}
\end{align*}
for any two points $\xi,\eta\in\Lambda_\Gamma$.
Now \cref{dense} does not require the above preparations, but follows easily from the description of bisectors in \cref{bisec}.
\end{rem}

Let now $X\subseteq H$ be a $\Gamma$-invariant subset of $H$ containing the origin.
Let $\mu=(\mu_x)_{x\in X}$ be a $\Gamma$-equivariant family of probability measures on $X$.
Assume that the associated random walk on $X$, also denoted  by $\mu$, is non-degenerate and transient with Green kernel $g=g(x,y)>0$, Martin kernel $k=k(x,y)$, and Martin boundary $\partial_\mu X$.
As in the introduction, we say that $\mu$ is an \emph{FLS-restriction} of $\mathcal Z$ to $X$ if the Martin kernels with origin some $o \in X$ satisfy $k(x,y)=K(x,y)$ for all $x,y \in X$ with $y\neq x, o$.

\begin{thm}\label{Martin}
Let $\mathcal{Z}$ be an Ancona process in $H$ and $(X,\mu)$ a $\Gamma$-invariant FLS-restriction of $\mathcal{Z}$, where $X\subseteq H$ is a $\Gamma$-invariant discrete subset.
Suppose that $\Lambda_X\subseteq\Lambda_H^{\rad}$.
Then there is a $\Gamma$-equivariant homeomorphism $\Lambda_X\to\partial_\mu X$ which extends the identity of $X$ continuously.
\end{thm}

\begin{proof}
Let $(x_n)$ be a divergent sequence in $X$, converging to a point $\xi$ in $\Lambda_X$.
Since $X\subseteq H$, $(x_n)$ converges to a point $\hat\xi\in\hat H=H\cup\partial_{\mathcal{Z}}H$, by \ref{geomar}.
Now $(X,\mu)$ is an FLS-restriction of $(H,\mathcal Z)$, hence the restriction $k(\cdot,\hat\xi)$ to $X$ of $K(\cdot,\hat\xi)$ is a Martin kernel in $\partial_\mu X$.
It follows that $\xi\to k(\cdot,\hat\xi)$ defines a continuous mapping $\Lambda_X\to\partial_\mu X$.
By \ref{geomar}, it is surjective, and, by \cref{dense}, it is injective.
Since both spaces are compact Hausdorff spaces, it is a homeomorphism.
\end{proof}

\begin{thm}\label{sullivan2}
In the situation of \cref{sullivan}, let $x\in H$.
Then LS-discretization of Sullivan's $\vf$-process $\mathcal S$ with respect to balanced LS-data induces an irreducible symmetric and $\Gamma$-invariant random walk $\mu$ on $X=\Gamma x$ with Martin boundary $\Gamma$-equivari\-antly homeomorphic to $\Lambda_\Gamma$.
Restriction defines an isomorphism from the space of $\lambda_0$-eigenfunctions of $\Delta$ on $H$ with growth of the order of $\varphi$ to the space of bounded $\mu$-harmonic functions on $X$.
Moreover, if $\vf$ is square-integrable, then $\mu$ has finite first moment and finite entropy.
\end{thm}

\begin{proof}
Except for symmetry, first moment, and entropy, the assertions are direct consequences of \cref{Martin}.
Symmetry of $\mu$ follows as in the proof of \cite[Theorem 2.7]{BL2}, observing that the Green kernel $g$ of $\mu$ is symmetric on $X$
since $\vf$, as the lift of a function on $M$, is constant on $X$.
For the penultimate assertion observe that $f\varphi$ is a $\lambda_0$-harmonic function of $\Delta$ in $H$ with growth $O(\varphi)$ if and only if $f$ is a bounded $(\Delta-2\nabla\ln\vf)$-harmonic function in $H$, which is, if and only if $f|_X$ is a bounded $\mu$-harmonic function. 

The assertions about first moment and entropy follow from \cref{fm} together with \cref{fmrem}.
\end{proof}

\begin{rem}\label{pinched2}
Assuming geometrical finiteness in the sense of Bowditch \cite{Bo}, it will be shown, in \cite{Nan2}, that $\mu$ has finite exponential moment if $\vf$ is square-integrable.
For cocompact $\Gamma$, this may be found in \cite[Lemma 3.13]{Ba} or \cite[Theorem 2.21]{BP3}. 
\end{rem}

\section{Relatively hyperbolic groups}\label{subrelhyp}
This section is an addendum to \cite{Nan}.
A finitely generated group $\Gamma$ is \emph{hyperbolic relative to a family $\Gamma_1,\dots,\Gamma_l$ of subgroups} if there is a proper Gromov hyperbolic geodesic space $H$ such that $\Gamma$ acts properly discontinuously by isometries on $H$ and such that there are open subsets $B_i\subseteq H$, $1\le i\le l$, with $\partial_{\geo}B_i=\partial B_i\cap\partial_{\geo}H=\{\xi_i\}$ such that
\begin{enumerate}
    \item the $B_i$ are \emph{precisely invariant}:
    for $g\in\Gamma$, we have $gB_i\cap B_j\ne\emptyset$ if and only if $i=j$ and $gB_i= B_i$ if and only if $i=j$ and $g\in\Gamma_i$;
    \item $\Gamma$ acts cocompactly on $H \setminus \cup_{1\le i\le l}\Gamma B_i$;
\end{enumerate}
compare with \cite{Bo1,GroMan}.
The stabilizers $\Gamma_{i}$ and their conjugates are called the \emph{maximal parabolic} subgroups of $\Gamma$.
If $H'$ is any other proper Gromov hyperbolic geodesic space, where $\Gamma$ acts with the above properties, then its Gromov boundary is $\Gamma$-equivariantly homeomorphic to $\partial_{\geo}H$. 
The space $\partial_{\geo}H$ is called the \emph{Bowditch boundary} of the data $(\Gamma,\{\Gamma_{1},\dots,\Gamma_{l}\})$.

\begin{exa}
If $\Gamma$ acts geometrically finitely on a Hadamard manifold with pinched negative curvature, then $\Gamma$ acts geometrically finitely on the convex hull $H$ of the limit set $\Lambda_\Gamma$ in the sense of Bowditch \cite{Bo1} and consequently is relatively hyperbolic.
The Bowditch boundary equals $\Lambda_\Gamma$. More generally, if $H$ is a proper, complete Gromov hyperbolic metric space and $\Gamma$ acts geometrically finitely on $\Lambda_\Gamma$ by a \textit{convergence action}, then $\Gamma$ is relatively hyperbolic; see \cite{Y}.
\end{exa}

\begin{exa}
Any (non-trivial) free product $P_1\ast\cdots\ast P_l$ of finitely many finitely generated groups is hyperbolic relative to the $P_i$ which are infinite, the product being hyperbolic if all $P_i$ are hyperbolic; see for example \cite{Osin}.
\end{exa}

\subsection{Normal subgroups of relatively hyperbolic groups}\label{secrh}
A reversible Markov chain $(\mathcal{Z},P,m)$ is a Markov chain $\mathcal Z$ in a graph $Y$ with stochastic matrix $P$ and a locally finite measure $m$ such that $m(y)P(y,y')=m(y')P(y',y)$, for all $y,y'\in Y$.

Let $\Gamma$ be a relatively hyperbolic group with virtually nilpotent parabolic subgroups of infinite cardinality, with a given finite symmetric generating set $S$, where $S$ is assumed to contain finite symmetric sets of generators of a set of representatives of the maximal parabolic subgroups. We denote below by $d_S$ the metric relative to $S$ in the Cayley graph of $\Gamma$.
In \cite{Nan}, a reversible $\Gamma$-Ancona Markov chain $(\mathcal{Z},P,m)$ is constructed in a proper, geodesic, hyperbolic graph $(H,\rho)$ together with an origin $o\in H$ such that the following conditions hold:
\begin{enumerate}
\item $\Lambda_\Gamma=\partial_{\geo} H$;
\item every point of $\partial_{\geo}H$ is a radial limit point for $H$;
\item $P,m$ are $\Gamma$-invariant and there is a constant $C>0$ such that $m(B_y)=C^{\pm1}$ for each $y\in H$ and $m(x)=C^{\pm 1}m(y)$, whenever $x,y \in H$ are joined by an edge;
\item $\p_o[Z_1= so]\geq c_0$, for each $s\in S$, for a constant $0<c_0<1$;
\item there is $C'\geq 1$ such that, if $x,y,z\in H$ and $x\neq y$, then $ G(x,y)\leq C'^{\,\rho(x,z)}\cdot G(z,y)$; moreover, $m(x)\cdot G(x,x)\leq C'^{\,\rho(x,y)}\cdot G(y,x)$ and $G(x,x)=C'^{\pm 1}$ for each $x$;
\item if $x\neq y$ and $\text{dist}(z,[x,y])\leq r$, then $$G(x,y)\leq  C'^{r+1}\cdot \frac{G(x,z)}{m(z)}\cdot G(z,y);$$
\item there are $\alpha>0, C''>0$  such that if $x\neq y$, $\frac{G(x,y)}{m(y)}\leq C''e^{-\alpha\rho(x,y)}$;
\item $\Gamma x$ is recurrent for $\mathcal{Z}$ with finite $j$-th hitting time $\tau_j$, for $j\in\N$, and an FLS-restriction of $\mathcal{Z}$ to $\Gamma$ given by $$\mu(go)=\p_o[Z_{\tau_1}=go];$$
\item there is a constant $c>0$, depending on $S$ such that $$\E_x\left[e^{c\cdot\rho(x,Z_{\tau_1})}\right]<\infty;$$
\item for any $g\in\Gamma$, there is a path $p_{eg}\subset H$, called a \emph{preferred path}, for which there exist maximal parabolic subgroups $(P_i)_{i=1}^k$ and horoballs $(B_i)_{i=1}^k$, precisely invariant with respect to the $P_i$ such that the difference $p_{eg}\backslash (\cup_{i=1}^k B_i)$ is contained in $\Gamma o$.
Moreover, denoting the first and last points of $p_{eg}\cap B_i$ by $a_io$ and $b_io$ with $a_i,b_i\in \Gamma$, we have \[2\rho(o,go)\geq\sum_{i=1}^k\rho(a_io,b_io)+L_S(p_{eg}\backslash \cup_{i=1}^k B_i),\]
where the second term on the right side above is understood as a sum where we add the $d_S$-length of all the connected subpaths (after identifying them with paths in $\Gamma$).
\end{enumerate}
The graph $H$ above is obtained by a slight variation of a construction of Groves-Manning \cite{GroMan}. For (1), (2), (3) and (4), see \cite[Section 2]{Nan}. For (5) and (6) see \cite[Section 3]{Nan}. For (7) see proof of \cite[Theorem 4.2]{Nan}. For (8) and (9) see \cite[Corollary 0.4]{Nan}. For (10), see \cite[Section 2]{Nan}.

Let $\Gamma'\trianglelefteq\Gamma$  be a normal subgroup with quotient homomorphism $q:\Gamma\to Q:=\Gamma/\Gamma'$. Let $A_j=q(Z_{\tau_j})$ and define inductively $T_0=0$ and \[T_j=\min\{i>T_{j-1}\mid A_i\neq A_{T_{j-1}}\}.\]
We get a random walk in $Q$, namely $(A_{T_{j}})_{j=0}^{\infty}$, with generating probability $$\mu_Q(w)=\p_x[A_{T_1}=w].$$ 

We specify a Cayley structure in the group $Q$ of deck transformations related to the one in $\Gamma$. For that purpose denote by $[g]$ the coset of $\Gamma'$ containing $g$ and define a length $l_Q$ in $(\Gamma,d_S)$ by setting for any path $\lambda=\cup_{i=1}^{j-1}[g_i,g_{i+1}]$ in $(\Gamma,S)$, $l_Q(\lambda)=\sum_{[g_i]\neq [g_{i+1}]}1$. Then define the metric $d_Q$ by letting for $g,h\in \Gamma$,
\[d_Q(q(g),q(h))=\min_{\lambda}l_{Q}(\lambda),\] where the minimum is over the set of paths $\lambda$ in $(\Gamma,d_S)$ joining a point in $[g]$ to a point in $[h]$.

\begin{lem}\label{proj1}
The set \[S_Q:=\{[\gamma]\in Q\mid l_Q(\lambda)=1,\, \text{for some}\;\lambda\;\text{joining}\;e,g;\; g\in [\gamma]\neq \Gamma'\}\] is a finite symmetric generating set. The left invariant Cayley metric of $(Q,S_Q)$ is $d_Q$.
\end{lem}
\begin{rem}\label{Lip} 
Note that $S_Q$ as defined above is the set of all $[\gamma]$ which are at $d_Q$-distance one of $\Gamma'$. Note also that $q$ is Lipschitz, that is, 
\[d_Q([\gamma_1],[\gamma_2])\leq d_S(\gamma_1,\gamma_2),\]
for any $\gamma_1,\gamma_2\in H$.
\end{rem}
\begin{proof}
Finiteness and symmetry of $S_Q$ follow, once it is checked that $[\gamma]\in S_Q$, if and only if $d_S(e,[\gamma])=1$ (as $d_S(e,[\gamma])=d_S(e,[\gamma^{-1}])$ because $\Gamma'$ is normal and $(\Gamma,d_S)$ is locally finite). If $[\gamma]\in S_Q$, then there are $h_1,h_2\in\Gamma'$, and a path $\lambda$ from $e$ to $\gamma h_2$, with $h_1$ the last point of $\Gamma'$ in $\lambda$, such that $h_1$, and $\gamma h_2$ are related in $S$. So $\gamma h_2=h_1s$, for some $s\in S$, so that $s=h_1^{-1}\gamma h_2=\gamma(\gamma^{-1}h_1\gamma)h_2$, so that $d_S(e,[\gamma])=1$. Next if $d_S(e,[\gamma])=1$, then there is $h\in \Gamma'$, such that $\gamma h=s\in S$. So $\gamma=sh^{-1}=(sh^{-1}s^{-1})s$, so that $[s]=[\gamma]$ and 
\[1=l_Q([sh^{-1}s^{-1},\gamma]_S)=l_Q([e,(shs^{-1})\gamma]_S)=l_Q([e,\gamma(\gamma^{-1}shs^{-1}\gamma)]_S),\] which means that $[\gamma]\in S_Q$.

Note that $S_Q$ generates $Q$ is clear. Indeed, for $g\in \Gamma$, there is a path $\lambda$ joining $e$ and $g'\in [g]$ such that $d_Q(e,[g])=l_Q(\lambda)$, so there are $g_i$ and $k$ such that $d_Q(e,[g])=\sum_{i=0}^{k-1}l_Q(\lambda_{g_i,g_{i-1}})$, for some path $\lambda_{g_{i},g_{i+1}}$ joining $g_i$ and $g_{i+1}$, where $[g_0]=e$ and $[g_k]=[g]$, and $l_Q(\lambda_{g_i,g_{i+1}})=1$ for each $i$. Hence, $l_Q(\lambda_{e,g_{k-1}^{-1}g_k})=1$, thus $g_k=g_{k-1}\gamma_{k}$, for some $[\gamma_k]\in S_Q$, so the claim follows.
\end{proof}

\begin{proof}[Proof of \cref{bowbou}]
If the parabolic subgroups are of finite cardinality, then $\Gamma$ is hyperbolic. In this case we take $\mathcal Z$ to be a simple random walk with respect to a finite symmetric generating set $S$, and still write $(H,\rho)$ for the Cayley graph $(\Gamma,d_S)$. It is easy to check that $\mathcal Z$ is an Ancona process.

Let us now call a subgroup maximal parabolic only if it is infinite and check that the reversible Markov chain $\mathcal Z$ described in the beginning of this section is an Ancona process. Note that by properties (3) and (5) of $\mathcal Z$, $G(x,B_y)\approx \frac{G(x,y)}{m(y)}$ when $x\neq y$.  Then (A1) follows from the last part of property (5) and from property (7) of $\mathcal Z$. Property (5) implies (A2) and property (6) implies (A3). The Martin boundary of $\mathcal Z$ is shown to be $\partial_{\geo}H$ in \cite[Section 3]{Nan}, which is (A4); by property (1) this is the Bowditch boundary.  

Now since $\Lambda_\Gamma'=\partial_{\geo}H$ (as $\Gamma'$ is normal), by \cref{Martin}, we only need to check that $\Gamma'$ is recurrent for $\mathcal{Z}$. This follows if the process $(A_{T_j})_{j=0}^\infty$ is recurrent in $Q$, and in this case $(Z_{T_j})_{j=1}^\infty$ is the required random walk. We check recurrence next.

We will show that $(Z_{T_j})_{j=1}^\infty$ has finite first (in fact exponential) moment with respect to $d_Q$.
First we claim that there is $b>0$ such that for $r>0$, \[d_Q(e_Q,q(g))\geq r\implies \rho(o,go)\geq b\cdot r.\]

When $\Gamma$ is hyperbolic, this follows from \cref{Lip}. To check this in general, write using notation of (10), to define a path in $(\Gamma,d_S)$, using $p_{eg}$ as follows. For $a_io$, $b_io$ the points of first entry and final exist from $B_i$ of $p_{eg}$, define the path in $\Gamma$ (identified with the corresponding path in $\Gamma o$) $\lambda_{eg}=\cup_{i=1}^k [a_io,b_io]_S\cup (p_{eg}\backslash \cup_{i=1}^k B_i)$. Then
\[d_Q(e,q(g))\leq \sum_{i=1}^k d_Q(q(a_i),q(b_i))+d_S(p_{eg}\backslash \cup_{i=1}^k [a_i,b_i]_S).\]
Note for $i\in\{1,\ldots,k\}$,
by the hypothesis that $P_i\cap \Gamma'$ is finite index in $P_i$, and the quasiconvexity of $P_i$ in $\Gamma$ (see \cite{DS} or \cite{GP}) and the description of $d_Q$ in terms of $l_Q$, that \[d_Q(q(a_i),q(b_i))\lesssim \diam_Q(q([a_i,b_i]_{P_i}));\]
and the latter is uniformly in $i$, bounded above (since there are finitely many conjugacy classes of maximal parabolic subgroups and $\Gamma'$ is normal), by, say $M$. So, $\sum_{i=1}^kd_Q(q(a_i),q(b_i))\leq M\sum_i\rho(a_io,b_io)$. In view of (10) above, the claim is verified.

Then we have (using the Chernoff bound) and property (4),
{\allowdisplaybreaks
\begin{align*}
\p_o[d_Q(e_Q,A_{T_1})\geq r]&=\sum_{j\ge 1} \p_o[d_Q(e_Q,A_j)\geq r, T_1=j]\\
&\leq \sum_{1\leq  j\leq ar}\p_o[d_Q(e_Q,A_j)\geq r] + \sum_{j\geq ar} \p_o[T_1=j] \\
& \leq \sum_{1\leq j\leq ar}\p_o[\rho(o,Z_{\tau_j})\geq b\cdot r] + \sum_{j\geq ar} \p_o[T_1=j]\\
&\leq  e^{-cbr}\sum_{1\leq j\leq ar} m^{j} + \sum_{j\geq ar} (1-c_0)^{j -1} \leq C\cdot e^{-\beta r},
\end{align*}}
for suitable $C>0$, $\beta>0$ obtained by choosing $a>0$ small. Thus $A_{T_1}$ has finite exponential moment for suitable range of exponents. Given the assumption on $Q$ then, $(A_{T_j})_{j=0}^\infty$ is recurrent as needed \cite[T1 on p.\,83]{Spi}. 
\end{proof}

\begin{exa}
Let $\F_n$ be the free group with $n\ge2$ generators and commutator subgroup $G=[\F_n,\F_n]$, where $\F_n/G\cong\Z^n$.
Any proper subgroup $G'$ of $\F_n$ containing $G$ is normal
and, by normality, an infinitely generated free group.
If $\F_n/G'$ is virtually $\Z$ or virtually $\Z^2$, one obtains a random walk $\mu$ on $G'$ such that $\partial_\mu G'=\partial_{\geo}\mathbb{F}_n$, the Gromov boundary of $\mathbb{F}_n$ with respect to the word metric, a Cantor set.
On the other hand, consider an action of $\F_n$ on the hyperbolic plane $H_\R^2$ with a surface of genus $g$ with $p$ punctures as quotient, where $n=2g+p-1$ and $p\ge1$ and where we call the generators of $\F_n$ correspondingly by \[a_1,b_1,\dots,a_g,b_g,c_1,\dots,c_{p-1}.\]
If $G'$ contains $c_1,\dots,c_{p-1}$,
then also $c_p=[a_1,b_1]\cdots[a_g,b_g]c_1^{-1}\cdots c_{p-1}^{-1}$, the loop around the $p$-th puncture.
We get a random walk $\mu'$ on $G'$ with respect to the associated relatively hyperbolic structure 
\[ (\F_n,\{\la c_1\ra,\dots,\la c_{p-1}\ra,\la c_p\ra\}) \]
such that $\partial_{\mu'}G'$ equals the Bowditch boundary of the relatively hyperbolic structure,
which is $\partial_{\geo}H_\R^2$, the circle.
\end{exa} 

\section{Non-harmonicity of boundary points}\label{secsem}
In this section, we need a small variation in the definition of Ancona processes, namely we replace \ref{decay} as follows.

\begin{dfn}\label{Procp}
We say that $\mathcal{Z}$ is an \emph{Ancona' process} if
\begin{enumerate}[leftmargin=*, label=(A\arabic*')]
\item\label{decayp}
for any $x \in H$ there are $\alpha>0$ and $C_x\ge 1$ such that for any $y \in \Gamma x$ with $\rho(x,y) \geq 2$ we have that
\begin{equation*}
    G(y,B_{x}) \le C_x e^{-\alpha \rho(x,y)}.
\end{equation*}
\end{enumerate}
\end{dfn}
\noindent
and if \ref{harnackin} -- \ref{geomar} hold.

\begin{rem}\label{ancp}
In our most important applications, $\Gamma$-invariant diffusion processes on manifolds, and the Markov chain $\mathcal Z$ of \cref{secrh}, both \ref{decay} and \ref{decayp}, are satisfied.
\end{rem}

Suppose now that $\mathcal Z$ is an Ancona' process on a proper $\delta$-hyperbolic space $H$ with Green and Martin kernels $G=G(x,y)$ and $K=K(x,\xi)$, respectively.
Let $\Gamma$ be a group which acts properly discontinuously and isometrically on $H$, preserving $\mathcal{Z}$.

\begin{dfn}\label{cococ}
We say that a point $\xi\in\partial_{\geo}H$ is a \emph{cuspidal limit point of $\Gamma$} and write $\xi\in\Lambda_\Gamma^{\cus}$ if the stabilizer $\Gamma_\xi$ of $\xi$ in $\Gamma$ is infinite and leaves a Busemann function $b$ centered at $\xi$ invariant.
Moreover, we require that there is a constant $c>0$ such that, for any $x\in H$ and $y\in\Gamma x$, $(y\xi)_{gx}\le c$ for some $g\in\Gamma_\xi$.
\end{dfn}

\begin{rem}
Let $\xi$ be a cuspidal limit point of $\Gamma$ and $(z_n)$ be a sequence converging to $\xi$ such that $b$ is the limit of the functions $b_n=\rho(.,z_n)-\rho(z_n,x)$.
Let $(g_m)$ be a sequence in $\Gamma_\xi$ such that $\rho(g_mx,x)$ diverges.
Then we have
\begin{align*}
    2(g_mx z_n)_x = \rho(g_mx,x) + \rho(z_n,x) - \rho(g_mx,z_n) =  \rho(g_mx,x) - b_n(g_mx).
\end{align*}
For $n\to\infty$, the left hand side tends to $2(g_mx\xi)_x$ and the right to $\rho(g_mx,x)-b(g_mx)$.
Since $b$ is invariant under $\Gamma_\xi$, we get that, as $m\to\infty$, $(g_mx\xi)_x$ tends to infinity and therefore that $(g_mx)$ tends to $\xi$.
In particular, $\xi\in\Lambda_\Gamma$, thus justifying the above notation $\Lambda_\Gamma^{\cus}$ for the set of cuspidal limit points.
\end{rem}

\begin{exa}\label{exaculipo}
If $H$ is a Hadamard manifold with pinched negative  curvature, $\Gamma$ is geometrically finite on $H$, and $\xi\in\partial_{\geo}H$ is a parabolic point (in the sense of Bowditch \cite{Bo}), then $\xi\in\Lambda_\Gamma^{\cus}$.
\end{exa}

\begin{lem}\label{coco}
For a cuspidal limit point $\xi$ of $\Gamma$ as above and any $R\in\R$, $b$ attains only finitely many values on $X$ below $R$.
\end{lem}

\begin{proof}
We can assume $b(x)=0$.
Now for any $y\in X \cap \{ b\le R\}$, there exists $g\in\Gamma_\xi$ such that $(y\xi)_{gx}\le c$.
Keeping in mind that $b(y)=b(y)-b(gx) \leq R$, we obtain from \cref{R2c} that $\rho(x,g^{-1}y) = \rho(gx,y) \le R  + 2c$. Denoting by $\bar{B}(x,R + 2c)$ the closed ball with center $x$ and radius $R+2c$, we readily see that $S = X \cap \bar{B}(x,R + 2c)$ is finite, $X$ being discrete. We deduce from the above that $g^{-1} y \in S$, and therefore, $b(y) = b(g^{-1}y) \in b(S)$. This means that $b$ attains only finitely many values on $X \cap \{  b \leq R\}$.
\end{proof}

\begin{thm}\label{nonharm}
Let $X=\Gamma x$ for some $x\in X$ and $\xi\in\Lambda_\Gamma^{\cus}\cap\Lambda_H^{\rad}$.
Suppose that $\Gamma_\xi\ne\Gamma$.
Let $b$ be a $\Gamma_\xi$-invariant Busemann function and $\mu$ an FLS-reduction of $\mathcal{Z}$ to $X$.
Then $k(\cdot,\xi)=K(\cdot,\xi)|_X$ is not $\mu$-harmonic.
\end{thm}

\begin{proof}[Proof for Sullivan's process on symmetric spaces]
If $H$ is a symmetric space of negative  curvature, then any Busemann function $b$ on $H$ centered at $\xi$ is invariant under $\Gamma_\xi$.
Denote by $b_\xi$ the Busemann function centered at $\xi$ with $b_\xi(o) = 1$.
By \cref{coco}, $b_\xi$ achieves its minimum on $X$, say at $x$.
Since $\Gamma_\xi\ne\Gamma$, there is a point $y\in X$ with $b_\xi(y)>b_\xi(x)$.
Now $K(z,\xi) = \exp(-\delta b_\xi(z))/\vf(z)$ for any $z \in H$, in view of \cref{SullMart}.
Since $\vf$ is $\Gamma$-invariant and $o \in X$, we derive that $\vf(z) = 1$ for any $z \in X$. Therefore,
\begin{align*}
    k(x,\xi) = e^{-\delta b_\xi(x)} > \sum_{z\in X}\mu_x(z)e^{-\delta b_\xi(z)} = \sum_{z\in X}\mu_x(z)k(z,\xi),
\end{align*}
which yields that $k(\cdot,\xi)$ is not $\mu$-harmonic.
\end{proof}

The general case is more elaborate and requires some preparations.

\begin{lem}\label{4d2c}
Suppose that an isometry $g$ of $H$ fixes a point $\xi\in\partial_{\geo}H$  and a Busemann function $b$ centered at $\xi$.
Let $(x_m)$ be a sequence in $H$  converging radially to $\xi$ with $\liminf(o\xi)_{x_m}=c$.
Then
\begin{align*}
    \limsup\rho(x_m,gx_m)\le4\delta+2c.
\end{align*}
\end{lem}

\begin{proof}
The proof of \cref{4d2c} is motivated by the arguments on \cite[top of p.\,42]{GH}. 
Let $(z_n)$ be a sequence in $H$ converging to $\xi$ such that the invariant Busemann function $b$ is the limit of the functions \[b_n=b_n(x)=\rho(x,z_n)-\rho(o,z_n)\] locally uniformly.
For $\ve$ given and $m$ sufficiently large, we then have
\begin{align*}
	\rho(z_n,x_m) + \rho(x_m,o)
	&= \rho(z_n,o) + \rho(z_n,x_m) + \rho(o,x_m) - \rho(o,z_n) \\
	&= \rho(z_n,o) + 2(oz_n)_{x_m}  \le \rho(z_n,o) + 2c + 2\ve.
\end{align*}
Hence radial convergence means that the triangle inequality is an equality up to the corresponding constant $2c$.
Since $g$ is an isometry fixing $\xi$, the sequence of $y_m=gx_m$ converges also radially to $\xi$ and hence
\begin{align*}
	\rho(z_n,y_m) + \rho(y_m,y_0)
	\le \rho(z_n,y_0) + 2c + 2\ve.
\end{align*}
By the above,
\begin{align*}
	2(x_my_m)_{z_n} &= \rho(x_m,z_n) + \rho(y_m,z_n) - \rho(x_m,y_m) \\
    &= \rho(z_n,o) - \rho(x_m,o) + 2(oz_n)_{x_m} \\
    &\hspace{11mm}+ \rho(z_n,y_0) - \rho(y_m,y_0) + 2(y_0z_n)_{y_m}  - \rho(x_m,y_m) \\
    &= 2(oy_0)_{z_n} + \rho(o,y_0) + 2(oz_n)_{x_m} + 2(y_0z_n)_{y_m} \\
    &\hspace{11mm}- \rho(x_m,o) - \rho(y_m,y_0) - \rho(x_m,y_m) 
\end{align*}
Since the last three terms in the penultimate line are bounded, the first two terms in the last line tend to $-\infty$, and the last term is negative,
we get that \[(oy_0)_{z_n}>(x_my_m)_{z_n}+2\delta\] for all sufficiently large $m$ and $n$.
Then
\begin{align*}
	(x_my_m)_{z_n} 
	&\ge \min\{(x_mo)_{z_n},(oy_0)_{z_n},(y_0,y_m)_{z_n}\} - 2\delta \\
	&\ge \frac12\min\{\rho(x_m,z_n)+\rho(o,z_n)-\rho(o,x_m), \\
	&\hspace{22mm}\rho(y_0,z_n)+\rho(y_m,z_n)-\rho(y_0,y_m)\} - 2\delta \\
	&\ge \min\{\rho(x_m,z_n),\rho(y_m,z_n)\} - 2\delta - c - \ve.
\end{align*}
Hence, with $d=2\delta + c + \ve$,
\begin{align*}
	\rho(x_m,y_m)
	&= \rho(x_m,z_n) + \rho(y_m,z_n) - 2(x_my_m)_{z_n} \\
	&\le \rho(x_m,z_n) + \rho(y_m,z_n) - 2\min\{\rho(x_m,z_n),\rho(y_m,z_n)\} + 2d\\
	&=  \rho(x_m,z_n) - \rho(o,z_n) + \rho(y_m,z_n) - \rho(o,z_n) \\
	&- 2\min\{\rho(x_m,z_n)-\rho(o,z_n),\rho(y_m,z_n)-\rho(o,z_n)\} + 2d.
\end{align*}
Now up to $\ve$, the first two terms on the right are equal to $b(x_m)=b(y_m)$
and the third to $2b(x_m)=2b(y_m)$, for all sufficiently large $n$.
This completes the proof of \cref{4d2c}.
\end{proof}

\begin{lem}\label{Minv}
Let $\xi\in\partial_{\geo}H$ be a radial limit point of $H$.
Suppose that an isometry $g$ of $H$ preserves $\xi$ and a Busemann function centered at $\xi$.
Then there is a constant $A>0$ such that
\begin{align*}
    K(x,\xi)/K(gx,\xi) = A^{\pm1} \quad\text{for any $x\in H$ and $g\in\Gamma_\xi$.}
\end{align*}
If $H$ is a Hadamard manifold with pinched negative  curvature, then $K(x,\xi)=K(gx,\xi)$.
\end{lem}

\begin{proof}
We have
\begin{align*}
    K(gx,\xi) &= K(gx,g\xi) = \lim_{m\to\infty}\frac{G(gx,gx_m)}{G(o,gx_m)} \\
    &= \lim_{m\to\infty}\frac{G(x,x_m)}{G(o,x_m)}\frac{G(o,x_m)}{G(o,gx_m)}
    = K(x,\xi)\lim_{m\to\infty}\frac{G(o,x_m)}{G(o,gx_m)}
\end{align*}
where the last factor is bounded by \ref{harnackin} and \cref{4d2c}.

If $H$ is a Hadamard manifold with pinched negative curvature, then we may choose $x_m=c(m)$ for some geodesic ray $c$ to $\xi$.
Then $\rho(x_m,gx_m)$ tends to $0$ by \cref{asym} and hence $G(o,x_m)/G(o,gx_m)$ to $1$ by the gradient estimate of Cheng and Yau \cite[Theorem 6]{CY}. 
\end{proof}

\begin{proof}[Proof of \cref{nonharm}]
Without loss of generality, we assume that $b(x)=0$.
 
Given $\ve>0$, there is $R\geq 2$ such that, for any $y\in X$, $b(y)>R$ implies that $K(y,\xi)<\ve$. To see this for any $y$ as above, let $g\in \Gamma_\xi$ be such that $(y\xi)_{gx}\leq c$ with $c$ as in \cref{cococ}.
Then, with $C_0$ as in \eqref{Mapprox} and $C_1=C_1(c)$ as in \ref{ancin},
\begin{align*}
    K(y,\xi)/K(gx,\xi) \le C_0 C_1 G(y,B_{gx}) = C_0 C_1 G(g^{-1}y,B_x).
\end{align*}
Now \ref{decayp} and \cref{Minv} yield
\begin{align*}
    K(y,\xi)\le AC_0C_1C_x e^{-\alpha R}K(x,\xi),    
\end{align*}
using that $\rho(g^{-1}y,x)>R\geq2$.

Suppose now that the restriction $f$ of $K(. ,\xi)$ to $X$ is $\mu$-harmonic.
Let $\ve>0$ be small and $R\ge2$ as above.
Then $f$ is bounded on $X\cap\{b\le R\}$, by \cref{coco}, \cref{Minv}, and \ref{harnackin}.
Furthermore, $f\le\ve$ on $X\cap\{b>R\}$.
Hence $f$ is a bounded $\mu$-harmonic function on $X$ and the supremum of $f$ is achieved in $X\cap\{b\le R\}$.
By \cref{coco}, there exists some $a\ge0$ such that the supremum of $f$ on $X$ is equal to the supremum $F$ of $f$ over $X \cap \{ b = a \}$.
Choose a point z on $X \cap \{ b = a \}$ and a sequence $g_n$ in $\Gamma_\xi$ such that the functions $f_n = f \circ g_n$ satisfy $f_n(z)\to F$.
Up to passing to a subsequence, $f_n$ converges pointwise to some $h$.
Using that the $f_n$ are $\mu$-harmonic and uniformly bounded, it follows from the dominated convergence theorem that $h$ is also $\mu$-harmonic.
Moreover, $h\leq F$, while $h(z) = F$, which implies that $h$ is constant, by the maximum principle.
However, since $g_n\in\Gamma_\xi$, it follows from 1) that $h$ attains arbitrarily small positive values, which is a contradiction.
\end{proof}

\section{Harmonicity of boundary points}\label{secharm}

Throughout this section, let $H$ be a Hadamard manifold of  curvature $-b^2 \leq K_H \leq -a^2 < 0$, $M = \Gamma \backslash H$ a non-compact, finite volume quotient and $p \colon H \to M$ the covering map. Choose $o \in F_0 \subset V_0 \subset M$ with $F_0$ closed and $V_0$ open and evenly covered, such that the preimages $F = p^{-1}(F_0)$ and $V = p^{-1}(V_0)$ constitute balanced LS-data for $X = p^{-1}(o)$. Finally, let $(\mu_y)_{y \in H}$ be the corresponding LS-measures. The goal of this section is to prove the following:

\begin{thm}\label{yesharm}
If $\xi \in \partial_{\geo} H$ is not a cuspidal point of $\Gamma$, then $K(\cdot,\xi)$ is swept by $F$. In particular, the restriction of $K(\cdot,\xi)$ to $X$ is $\mu$-harmonic.
\end{thm}

The following auxiliary result is motivated by \cite[Theorem 3.1]{BPe}.

\begin{thm}\label{sweep preim comp}
Let $p \colon N_1 \to N_0$ be a normal Riemannian covering and $D \subset N_0$ a compact domain with $C^2$ boundary. Then any positive harmonic function $h$ on $N_1$ satisfies $h(x) = \beta_x(h)$ for any $x \in p^{-1}(D)$, where $\beta$ stands for the balayage onto $N_1 \setminus p^{-1}(D^\circ)$.
\end{thm}

\begin{proof}
Fix $o \in D^\circ$ and set $X = p^{-1}(o)$. Consider the Dirichlet fundamental domains
\[
\tilde{D}_y = \{ x \in N_1 : d(x,y) \leq d(x,z) \text{ for any } z \in  X\}
\]
with $y \in X$, and let $D_y = \tilde{D}_y \cap p^{-1}(D)$. It should be noticed that $D_y$ is compact and $D_{gy} = g D_y$ for any $y \in X$ and $g \in \Gamma$, where $\Gamma$ stands for the deck transformation group of $p$. Choose a partition $(U_y)_{y \in X}$ of $p^{-1}(D)$ with $U_y \subset D_y$ for any $y \in X$. Choose $y \in X$ and a compact domain $W_y \subset p^{-1}(D)$ with piecewise $C^2$ boundary containing a neighborhood (in $p^{-1}(D)$) of $D_y$. 
The existence of such $W_y$ follows for instance by considering an exhaustion $f \in C^\infty(2M)$ of the doubled manifold $2M$ (obtained by gluing two copies of $M:=p^{-1}(D)$ along their boundary) which is invariant under reflection along $\partial M$, and setting $W_y = M \cap \{ f \leq c \}$ for a sufficiently large regular value $c$ of $f$, where we note that $\{f = c\}$ intersects $\partial M$ transversally. Set $F_y = W_y \cap \partial p^{-1}(D)$ and $V_y = \partial W_y \setminus F_y$. 

Then there exists $c \geq 1$ such that
\begin{equation}\label{est}
\frac{h(z_1)}{h(z_2)} = c^{\pm 1} \text{ and } \beta^{W_y}_{z_1}(F_y) \geq c^{-1}
\end{equation}
for any $z_1,z_2 \in D_y$, where $\beta^{W_y}$ stands for the balayage onto $N_1 \setminus W_y^\circ$. Indeed, the first assertion follows from the gradient estimate of Cheng and Yau \cite[Theorem 6]{CY} after noticing that there exists $\delta > 0$ such that the $\delta$-tubular neighborhood $D_\delta$ of $D$ is compact and $h$ is harmonic in the preimage $p^{-1}(D_\delta)$ which is the $\delta$-tubular neighborhood of $p^{-1}(D)$. As for the second assertion, consider a continuous function $f$ on $\partial W_y$ which is one on $\partial W_y \cap D_y$ and zero on $V_y$. Then $\beta^{W_y}_{z_1}(F_y) \geq \beta^{W_y}_{z_1}(f) = \tilde{f}(z_1)$ for any $z_1 \in W_y$, where $\tilde{f}$ stands for the harmonic extension of $f$ in $W_y$. This establishes the assertion, since $\tilde{f}$ attains positive minimum over $W_y \cap D_y$.

For any $g \in \Gamma$, let $W_{gy} = gW_y$, $F_{gy} = g F_y$ and $V_{gy} = g V_y$. Then it is evident that (\ref{est}) holds for any $z_1,z_2 \in U_y$ and any $y \in X$. For any $x \in p^{-1}(D)$ there exists unique $y \in X$ such that $x \in U_y$, and since $W_y$ is compact, it follows that 
\[
h(x) = \beta_x^{W_y}(h) = \beta_x^{W_y}(h|_{F_y}) + \beta_x^{W_y}(h|_{V_y}) 
\]
where
\[
\beta_x^{W_y}(h|_{F_y}) \geq c^{-2} h(x),
\]
in view of (\ref{est}). Given any measure $\mu$ on $U_y$, let
\[
\beta_{\mu} = \int_{U_y} \beta_z^{W_y} \mu(dz).
\]
More generally, if $\mu$ is a measure on $p^{-1}(D)$, denote by $\mu_{y}$ with $y \in X$ its restriction to $U_y$, and set
\[
\beta_\mu = \sum_{y \in X} \beta_{\mu_y}.
\]
Given any such measure $\mu_n$ with $\mu_n(h) < + \infty$, we obtain that
\[
\mu_n(h) = \tau_{n+1}(h) + \mu_{n+1}(h),
\]
where 
\begin{eqnarray}
\tau_{n+1}(h) &=& \sum_{y \in X} \beta_{\mu_n}(h|_{F_y}) \geq c^{-2} \mu_n(h), \nonumber \\
\mu_{n+1}(h) &=& \sum_{y \in X} \beta_{\mu_n}(h|_{V_y}) \leq (1-c^{-2}) \mu_n(h). \nonumber
\end{eqnarray}
For any $x \in p^{-1}(D^\circ)$, let $\mu_0 = \delta_x$. It follows that
\[
h(x) = \sum_{i=1}^{n} \tau_i(h) + \mu_{n+1}(h)
\]
for any $n \in \mathbb{N}$, where $\mu_{n+1}(h) \leq (1 - c^{-2})^n h(x) \rightarrow 0$. We conclude that
\[
h(x) = \beta_x(h)
\]
for any $x \in p^{-1}(D^\circ)$, as we wished.
\end{proof}

\begin{lem}\label{mean value prop}
Let $\eta \in \partial_{\geo} H$ be a cuspidal point of $\Gamma$ and $B$ a precisely invariant horoball centered at $\eta$. Then any bounded harmonic function $h$ on $B$ satisfies $h(x) = \bar{\beta}_x(h)$ for any $x \in B$, where $\bar{\beta}_x$ stands for the balayage onto $H \setminus B^\circ$.
\end{lem}

\begin{proof}
Since $\Gamma \backslash H$ is recurrent, it follows that $H \setminus B^\circ$ is recurrent, that is, $\bar{\beta}_x(H \setminus B^\circ) = 1$ for any $x \in B$. Therefore, $h$ is swept by $H \setminus B^\circ$, bounded harmonic functions being swept by recurrent subsets.
\end{proof}

\noindent{\emph{Proof of \cref{yesharm}.}} Fix a union of cusps $C \subset M$ about all parabolic points, disjoint from the closure of $V_0$. 
Write $C_0 = p^{-1}(C) \subset H$ as the union of disjoint horoballs $\{ b_n \leq 0 \}$ with $n \in \mathbb{N}$, where $b_n$ are Busemann functions centered at parabolic points of $\Gamma$. Consider also the union $C_1$ of $\{b_n \leq -1\}$ with $n \in \mathbb{N}$. It follows from Theorem \ref{sweep preim comp} for $D = M \setminus (p(C_1) \cup F)^\circ)$ that the Martin kernel $h = K(\cdot , \xi)$ satisfies
\[
h(x) = \beta_x(h) = \beta_x(h|_F) + \beta_x(h|_{C_1})
\]
for any $x \in H \setminus C_0^\circ$, where $\beta_x$ stands for the balayage onto $C_1 \cup F$. Since $\xi$ is not a cuspidal point of $\Gamma$, we deduce from \cite[Theorem 3]{An} that $h$ is bounded on each connected component of $C_1$. Therefore,  \cref{mean value prop} gives that
\[
\beta_x(h|_{C_1}) = \int_{\partial C_1} h(y) \beta_x(dy) = \int_{\partial C_1} \int_{\partial C_0} h(z) \bar{\beta}_y(dz) \beta_x(dy),
\]
where $\bar{\beta}$ stands for the balayage onto $H \setminus C_0^\circ$. This may be rewritten as
\[
h(x) = \tau_1(h) + \mu_1(h),
\]
where $\tau_1(h) = \beta_x(h|_F)$ and $\mu_1 = \int_{\partial C_1} \bar{\beta}_y \beta_x(dy)$. For any $y \in X$ denote by $D_y$ the Dirichlet fundamental domain centered at $y$. Then there exists $c > 1$ such that
\[
\frac{h(z_1)}{h(z_2)} = c^{\pm 1} \text{ and } \beta_{z_1}(F_y) \geq c^{-1}
\]
for any $z_1, z_2 \in D_y \setminus C_0^\circ$ and any $y \in X$, where $F_y$ stands for the connected component of $F$ containing $y$, since $D_y \setminus C_0^\circ$ is compact and away from $C_1 \cup (F \setminus F_y)$. This implies that
\[
\tau_1(h) \geq \beta_x(h|_{F_y}) \geq c^{-2} h(x) \text{ and } \mu_1(h) \leq (1 - c^{-2}) h(x)
\]
for any $x \in D_y \setminus C_0^\circ$ and $y \in X$.

Given any measure $\nu$ supported in $M \setminus C_0^\circ$ with $\nu(h) < +\infty$, it follows that
\[
\nu(h) = \tau_{\nu}(h) + \mu_{\nu}(h),
\]
where $\tau_{\nu}(h) = \beta_{\nu}(h|_F)$ and $\mu_{\nu} = \int_{\partial C_1} \bar{\beta}_y \beta_{\nu}(dy)$ satisfy
\[
\tau_{\nu}(h) \geq c^{-2} \nu(h) \text{ and } \mu_{\nu}(h) \leq (1 - c^{-2}) \nu(h).
\]
For $\mu_0 = \delta_x$ and $\tau_0 = 0$, and $\mu_n(h) = \tau_{n+1}(h) + \mu_{n+1}(h)$ for $n \in \mathbb{N} \cup \{0\}$, we obtain that
\[
h(x) = \sum_{i \leq n} \tau_i(h) + \mu_{n+1}(h),
\]
for any $x \in H \setminus C_0^\circ$ and $n \in \mathbb{N}$, where $\mu_{n+1}(h) \leq (1-c^{-2})^{n} h(x) \rightarrow 0$ as $n \rightarrow + \infty$. This yields that $h(x) = \beta_x^F(h)$ for any $x \in H \setminus C_0^\circ$. We conclude that this holds for any $x \in H$ and thus, $h|_X$ is $\mu$-harmonic (cf. \cite[Theorem 1.10]{BL2} and the references therein). \qed

\appendix
\section{Hyperbolic planes in symmetric spaces}\label{symm}
In this appendix, our presentation follows the standard one in the theory of symmetric spaces.

Let $S$ be a symmetric space of noncompact type, represented by a symmetric pair $(G,K)$,
where $G$ is a semisimple Lie group which acts almost effectively on $S=G/K$.
Let $s$ be a corresponding involution of $G$, $\theta=s_{*e}$,
and $\mathfrak g=\mathfrak k\oplus\mathfrak p$ be the associated Cartan decomposition of the Lie algebra of $G$.
Endow $\mathfrak g$ with an $\Ad_G$-invariant bilinear form $\la.,.\ra$,
which is negative definite on $\mathfrak k$ and positive definite on $\mathfrak p$,
such that the identification $T_xS=\mathfrak p$ is an orthogonal transformation, where $x=K/K\in S$.
Then $\ad_X$ is a symmetric transformation of $\mathfrak g$ with respect to the positive definite inner product $-\la\theta.,.\ra$,
for all $X\in\mathfrak p$.

Let $\mathfrak a$ be a maximal abelian subalgebra of $\mathfrak p$.
For a linear map $\alpha\colon\mathfrak a\to\R$, let
\begin{align*}
	\mathfrak g_\alpha = \{ X\in\mathfrak g \mid \text{$[H,X] = \alpha(H)X$ for all $H\in\mathfrak a$} \}. 
\end{align*}
Then $\mathfrak g_0=\mathfrak z_{\mathfrak k}\mathfrak a \oplus \mathfrak a$ and,
since all $\ad_H$, $H\in\mathfrak a$, are symmetric with respect to the above inner product and pairwise commuting,
\begin{align*}
	\mathfrak g = \mathfrak g_0 \oplus \sum_{\alpha\in\Delta} \mathfrak g_\alpha,
\end{align*}
where the set $\Delta$ of \emph{roots}  consists of those $\alpha\ne0$ for which $\mathfrak g_\alpha\ne\{0\}$.
For any root $\alpha$, we have the corresponding \emph{root vector} $H_\alpha\in\mathfrak a$ such that $\alpha=\la H_\alpha.,.\ra$.
Straightforward computations show that
\begin{align*}
	\theta \mathfrak g_\alpha = \mathfrak g_{-\alpha}
	\quad\text{and}\quad 
	[\mathfrak g_\alpha,\mathfrak g_\beta] \subseteq \mathfrak g_{\alpha+\beta}.
\end{align*}
In particular, $\Delta=-\Delta$.
Choosing a partition $\Delta=\Delta^+\cup\Delta^-$, we get
\begin{align*}
	\mathfrak g = \mathfrak z_{\mathfrak k}\mathfrak a \oplus \mathfrak a
	\oplus \sum_{\alpha\in\Delta^+} \mathfrak k_\alpha \oplus \mathfrak p_\alpha,
\end{align*}
with
\begin{align*}
	\mathfrak k_\alpha = (\mathfrak g_\alpha \oplus \mathfrak g_{-\alpha}) \cap \mathfrak k
	\quad\text{and}\quad
	\mathfrak p_\alpha = (\mathfrak g_\alpha \oplus \mathfrak g_{-\alpha}) \cap \mathfrak p.
\end{align*}
For any $\alpha\in\Delta^+$, $\dim\mathfrak k_\alpha=\dim\mathfrak p_\alpha$.
Furthermore, for any $X\in\mathfrak k_\alpha$,
there is a unique $Y\in\mathfrak p_\alpha$ such that $X+Y\in\mathfrak g_\alpha$.
We say that $X$ and $Y$ are \emph{related}.
Related vectors satisfy $|X|=|Y|$ and
\begin{align*}
	[H,X]=\alpha(H)Y \quad\text{and}\quad [H,Y]=\alpha(H)X,
\end{align*}
for all $H\in\mathfrak a$.
Moreover, $[X,Y]=|X||Y|H_\alpha$.

\begin{prop}\label{hypla}
For any $\alpha\in\Delta^+$ and unit vector $Y\in\mathfrak p_\alpha$,
$H_\alpha$ and $Y$ span a totally geodesic hyperbolic plane in $S$ through $x$ of curvature $-|\alpha|^2$.
In particular, for $y=\exp(sY)$ and $z=\exp(tH_\alpha)$,
\begin{align*}
	\cosh(|\alpha|d(y,z)) = \cosh(|\alpha|s)\cosh(|\alpha|t).
\end{align*}
\end{prop}

\begin{proof}
By what we said above, the linear hull $\mathfrak q\subseteq\mathfrak p$ of $H_\alpha$ and $Y$ satisfies
\begin{align*}
	[\mathfrak q,[\mathfrak q,\mathfrak q]] \subseteq \mathfrak q.
\end{align*}
Hence $\mathfrak q$ spans a totally geodesic plane through $x$.
Since
\begin{align*}
	|\alpha|^2=|H_\alpha|^2=\alpha(H_\alpha)
	\quad\text{and}\quad
	- R(Y,H_\alpha)H_\alpha = [H_\alpha,[H_\alpha,Y]] = |\alpha|^4 Y,
\end{align*}
we obtain the assertion about the curvature.
The assertion about the distance follows from hyperbolic trigonometry,
where we note that $H_\alpha$ and $Y$ are perpendicular.
\end{proof}

\begin{rem}\label{hyplar}
Since $R(Y,H)H=-[H,[H,X]]$, the spaces $\mathfrak p_\alpha$ correspond to the eigenspaces of the curvature endomorphisms $R(.,H)H$,
for $H\in\mathfrak a$.
\end{rem}

\section{Bisectors in hyperbolic geometry}\label{bisec}

In a metric space $H$, the \emph{bisector} of points $y,z\in H$ is the set \[\{x\in H\mid d(x,y)=d(x,z)\}.\]
In this section, we give a description of bisectors in hyperbolic geometry.
The description is certainly known to experts and is trivial in the case of real hyperbolic spaces.
However, we were not able to locate a presentation for quaternionic hyperbolic spaces and the octonionic hyperbolic plane.
We follow the nice account in \cite[Section 5]{Go}, where the case of complex hyperbolic spaces is treated.
Our main contribution is \cref{hypla} respectively its consequence, \cref{hypla2}.

Let $H=H^k_\F$ with $\F\in\{\C,\H,\O\}$ and normalize the standard metric on $H$ such that its maximal  curvature is $-1$.
Set $d=\dim_\R\F$ so that the (real) dimension of $H^k_\F$ is equal to $m=kd$.

Let $v$ be a unit tangent vector of $H$ at a point $x\in H$.
Then there is a unique $d$-dimensional totally geodesic submanifold $L\subseteq H$ passing through $x$ with $T_xL=V$,
where $V\subseteq T_xH$ is the linear hull of $v$ and the subspace $U_2$ consisting of vectors $u\in T_xH$ perpendicular to $v$ such that $R(u,v)v=-4u$.
Metrically, $L\cong H_\R^d$, but with constant  curvature $-4$.
The orthogonal complement $U_1$ of $V$ in $T_xH$ consists of vectors $u\in T_xH$ such that $R(u,v)v=-u$. 

\begin{prop}\label{hypla2}
Any $u\in U_1$ determines a totally geodesic real hyperbolic plane $P$ through $x$ of curvature $-1$ such that $T_xP$ is spanned by $u,v$.
\end{prop}

\begin{proof}
With respect to the setup in \cref{symm}, we choose $\mathfrak a=\R v$.
Then $U_1$ corresponds to a restricted root space $\mathfrak p_1$ with $R(u,v)v=-u$ for all $u\in U_1$.
Now the assertion follows from the first part of \cref{hypla}.
\end{proof}

Denote by $\pi\colon H\to L$ be the orthogonal, that is, nearest point projection.

\begin{prop}\label{hypla4}
Let $y\in L$ and $z\in H$. Then
\begin{align*}
    \cosh d(y,z) = \cosh d(y,\pi z)\cosh d(\pi z,z)
\end{align*}
\end{prop}

\begin{proof}
We can assume that $\pi z\ne y$ and that $z\notin L$.
Then we let $v\in T_{\pi z}L$ be the unit tangent vector pointing at $y$ and $u\in T_{\pi z}H$ the one pointing at $z$.
Then $u$ is perpendicular to $T_{\pi z}L$, and hence $u\in U_1$ in the above notation.
The formula for the distances is now just the formula in \cref{hypla},
since the real hyperbolic plane $P$ in \cref{hypla2} has curvature $-1$.
\end{proof}

For all $y_1\ne y_2$ in $H$, there is a unique totally geodesic subspace $L=L_{y_1,y_2}$ as above containing both of them.

\begin{cor}\label{bisec2}
The bisector of $y_1\ne y_2$ in $H$ is the preimage under the nearest point projection onto $L$ as above of the bisector of $y_1,y_2$ in $L$.
\end{cor}

\begin{proof}
This follows readily from Proposition \ref{hypla4}.
\end{proof}

\section{On first moment and entropy of LS-discretizations}\label{subm} 
The goal of this appendix is to establish the last assertion of the first part of \cref{main1}.

\begin{thm}\label{fm}
Suppose that $\mathcal Z$ is a regular process in $H$ and that there exist  $d , \alpha > 0$ such that
\[
\frac{\mathcal{L}(B(x,r))}{\mathcal{L}(B(x,s))} \leq \big(\frac{r}{s}\big)^d e^{\alpha r}
\]
for any $x \in H$ and $r \geq s > 0$.
Suppose further that ${Z_t}_\ast\p_x$ has density $p(t,x,y)$ with respect to $\mathcal L$ such that, for some $C>0,a>0,n_0>0$
and all $x,y\in H$ and $0 < t \leq 1$,
\begin{equation}\label{G0}
 p(t,x,y)\leq C\mathcal{L}(B(x,\sqrt{t}))^{-1}\left(1+\frac{\rho(x,y)^2}{t}\right)^{n_0} e^{-\frac{\rho(x,y)^2}{at}}.
\end{equation}
Let $\varphi$ be a positive $\lambda$-eigenfunction for some $\lambda>0$; that is, for all $t>0$,
\begin{align*}
    \varphi(x) = \int e^{\lambda t}p(t,x,y)\varphi(y)\,d\mathcal{L}(y),
\end{align*}
and let $\mathcal{Z}_\vf=(Z_{\vf t})$ be the associated $\vf$-process with transition density
\begin{align*}
    p_\vf = p_\vf(t,x,y) = e^{\lambda t}p(t,x,y)\vf(y)/\vf(x).
\end{align*}
Suppose that $\mathcal{Z}_\varphi$ is recurrent modulo a discrete group of isometries $\Gamma$ acting properly discontinuously on $H$
and that there is $A>0$ such that $\varphi(x)\leq A\varphi(y)$ for all $x,y\in H$ with $\rho(x,y)\le1$.
Then any LS-discretization $\mu$ of $\mathcal{Z}_\varphi$ to a $\Gamma$-orbit has finite first moment and finite entropy if also the stopping time associated to $\mu$ has finite first moment. 
\end{thm}

\begin{rem}\label{fmrem}
Recall that bounds on ratios of volumes of balls, on $p$ and $\varphi$ as in \cref{fm} hold when $H$ is a complete Riemannian manifold
with Ricci curvature bounded from below. 
In this situation the Sullivan's process $\mathcal{Z}_\varphi$ associated to the Laplacian of $H$ is regular and the heat kernel $p$ satisfies \eqref{G0} (see  \cite[Theorem 4.2]{S-C}). The Harnack inequality for $\varphi$ follows \cite[Theorem 6]{CY}.
\end{rem}

\begin{proof}[Proof of \cref{fm}]
Note that the density
\begin{align*}
   & p_\varphi(t,x,y)\\ &\hspace{0.5cm}\le C(\mathcal{L}(B(x,\sqrt{t})))^{-1} \left(1+\frac{\rho(x,y)^2}{t}\right)^{n_0}\exp\left(-\frac{\rho(x,y)^2}{at}+A\rho(x,y)+\lambda t\right)
\end{align*}
for all $x,y \in H$ and $0 < t \leq 1$ by \eqref{G0}.
It is not hard to verify that there exists $C>0$ such that $\E_y[\rho(y,Z_s)] \leq C$ for any $y \in H$ and $0 \leq s \leq 1$.
The claim about finite moment now follows from \cref{linmom}.
The claim about finite entropy follows by using that the first moment is finite as in \cite[Lemma 2.1]{BL}.
\end{proof}

\begin{prop}\label{linmom}
Let $\mathcal Z$ be a strong Markov process in $H$ and $\tau$ a stopping time adapted to $\mathcal Z$.
Then we have, for any $x\in H$, 
$$\E_x[\rho(x,Z_\tau)]\leq 3(2+\E_x[\tau])\cdot \sup_{\substack{y\in H\\0\leq s\leq 1}}\E_y[\rho(y,Z_s)].$$
\end{prop}

\begin{proof}
We may assume that \[C:=\sup_{\substack{y\in H\\0\leq s\leq 1}}\E_y[\rho(y,Z_s)]<\infty.\]
By the triangle inequality, the expectation $\E_x[\rho(x,Z_\tau)]$ is bounded from above by
\begin{equation*}
    \sum_{n=0}^\infty \left[\sum_{i=0}^{n-1} \int_{\{\tau\in(n,n+1]\}}\rho(Z_i,Z_{i+1}) \,d\p_x+\int_{\{\tau\in(n,n+1]\}}\rho(Z_n,Z_\tau)\,d\p_x\right] .
\end{equation*}
Write $\p_x^E=\p_x|_E$, for the restriction to a subset $E$ of paths. We first check
\begin{align*}
&\sum_{n=0}^\infty \sum_{i=0}^{n-1}\int_{\{\tau\in(n,n+1]\}}\rho(Z_i,Z_{i+1}) \,d\p_x \\
&= \sum_{i=0}^\infty \sum_{n=i+1}^{\infty}\int_{\{\tau\in(n,n+1]\}}\rho(Z_i,Z_{i+1}) \,d\p_x \\
&=\sum_{i=0}^\infty \int_{\{\tau> i + 1\}}\rho(Z_i,Z_{i+1}) \,d\p_x \\
&\leq\sum_{i=0}^\infty  \int_H \E_w[\rho(w,Z_1)]\,d\left[Z_{i\ast}\p_x^{\{\tau>i\}}\right](w) \\
&\leq C\sum_{i=0}^\infty \p_x[\{\tau>i\}]
\leq C(1+\E_x[\tau]).
\end{align*}
Set $\overline{Z}_\tau:=(Z_\tau,\tau)$. We check
{\allowdisplaybreaks
\begin{align*}
\sum_{n=0}^\infty &\int_{\{\tau\in(n,n+1]\}}\rho(Z_n,Z_\tau)\,d\p_x \\
&= \sum_{n=0}^\infty \int_0^\infty \p_x[\{\rho(Z_n,Z_\tau)>t\}\cap \{\tau\in(n,n+1]\}]\,dt \\
& \leq   \sum_{n=0}^\infty \int_0^\infty \int_{H\times (n,n+1)} \p_y[\{ \rho(y,Z_{ n+1-s})>t/2 \}]\,d\left[\overline{Z}_{\tau\ast}\p_x\right](y,s)\,dt \\
&\hspace{20mm}+ \sum_{n=0}^\infty \int_0^\infty \int_H  \p_y[\{\rho(y,Z_1)>t/2\}] d\left[Z_{n\ast}\p_x^{\{\tau>n\}}\right](y)\,dt \\
& \leq 2C \sum_{n=0}^\infty  \p_x[\{\tau\in(n,n+1)\}]\,+\, 2C\sum_{n=0}^\infty \p_x[\{\tau>n\}] \\
&\leq 2C(2+\E_x[\tau]).
\end{align*}}
The asserted inequality follows from the two estimates.
\end{proof}

\begin{lem}\label{ffmtau} Let $\mu$ be an LS-discretization of a regular Markov process $\mathcal Z$, as described in \cref{ssecls}.
Let $S_n$ and $\tau$ be as in\eqref{rs} and \eqref{tau}.
Then for $x\in V$,
\[\E_x[\tau]\leq \left(\sum_{j=0}^\infty(j+1)(1-C^{-2})^j\right)^2 \left(1+ \left(\max_{y\in \partial V}\E_y[S_1]\vee \E_x[S_1]\right)\right),
\]
where $C$ is the constant from \eqref{d4}, \eqref{nk}.
\end{lem}
\begin{proof}
By construction of $\tau$, we have for $w\in \overline{V}$, $n\in\N$, \begin{equation}\label{LS}\hat{\p}_w[\tau>S_n]\leq (1-C^{-2})^n.\end{equation}
We have
\begin{equation*}
\begin{split}
\int \tau \,d\hat{\p}_x & \leq \sum_{n=1}^\infty \int S_n\cdot \chi_{\{\tau=S_n\}}\,d\hat{\p}_x
= \sum_{n=1,k=0}^\infty \hat{\p}_x[S_{n}>k,\tau=S_n].
\end{split}
\end{equation*}
Recall that $\overline{Z}_\tau=(Z_\tau,\tau)$ and set $$m_{v,j}:=Z_{S_j\ast}\hat{\p}_v^{\{\tau>S_j\}}\quad \text{and}\quad \overline{m}_{v,j}:=\overline{Z}_{S_j\ast}\hat{\p}_v^{\{\tau>S_j\}}.$$
Then we have by \eqref{LS},
{\allowdisplaybreaks
\begin{align*}
\sum_{\substack{k=0\\ n=1}}^\infty &\hat{\p}_x[S_{n}>k,\tau=S_n]
\leq \sum_{\substack{k=0\\ n=1}}^\infty\sum_{i=0}^{n-1} \hat{\p}_x[S_{i+1}-S_i>k/n,\tau=S_n]\\
&= \sum_{\substack{i=0,k=0\\ n=i+1}}^\infty \int_H \left[\int_{H\times (k/n,\infty)} \hat{\p}_w[\tau=S_{n-i-1}]\,d\overline{m}_{y,1}(w,s)\right]\,dm_{x,i}(y) \\
&\leq \sum_{\substack{i=0,k=0\\ n=i+1}}^\infty  (1-C^{-2}) ^{n-i-1}\int_H \hat{\p}_y[S_1>k/n]\, dm_{x,i}(y)\\
&\leq \sum_{\substack{i=1\\ n=i+1}}^\infty  (1-C^{-2})^{n-i-1}\E_{m_{x,i}}[n(1+\E_{(\cdot)} [S_1])] \\
&\hspace{44mm}+\sum_{n=1}^\infty (1-C^{-2})^{n-1}\,n(1+\E_x[S_1])\\
&\leq (1+\max_{y\in\partial V}\E_y[S_1])\sum_{\substack{i=1\\ n=i+1}}^\infty  n(1-C^{-2})^{n-i-1} \hat{\p}_x[\tau>S_i]] \\
&\hspace{44mm}+\sum_{n=1}^\infty (1-C^{-2})^{n-1}\,n(1+\E_x[S_1])\\
&\leq \left(\sum_{j=0}^\infty (j+1)(1-C^{-2})^{j}\right)^2\left(1+ \left(\max_{y\in \partial V}\E_y[S_1]\vee \E_x[S_1]\right)\right),
\end{align*}
}
where the last inequality used another application of \eqref{LS}.
\end{proof}

\newpage


\begin{thebibliography}{xxxx}

\bibitem{An}
 A.\,Ancona,
 Negatively curved manifolds, elliptic operators, and the Martin boundary.
 \emph{Ann. of Math.} {\bf 125} (1987), no.~3, 495--536.
 
\bibitem{An1} 
---{}---{}---,
Positive harmonic functions and hyperbolicity.
\emph{Potential theory--surveys and problems (Prague, 1987)}, 1--23, Lecture Notes in Math. 1344, Springer, Berlin, 1988.

\bibitem{AS}
M.\,T.\,Anderson and R.\,Schoen,
Positive harmonic functions on complete manifolds of negative curvature.
\emph{Ann. of Math.} {\bf 121} (1985), no.~3, 429--461. 
\bibitem{Ba}
W.\,Ballmann,
On the {D}irichlet problem at infinity for manifolds of nonpositive curvature.
\emph{Forum Math.} {\bf 1} (1989), no.~2, 201--213.

\bibitem{BL}
W.\,Ballmann and F.\,Ledrappier,
The {P}oisson boundary for rank one manifolds and their cocompact lattices.
\emph{Forum Math.} {\bf 6} (1994), no.~3, 301--313.

\bibitem{BL2}
---{}---{}---,
Discretization of positive harmonic functions on Riemannian manifolds and Martin boundary.
\emph{Actes de la Table Ronde de G\'eom\'etrie Diff\'erentielle} (Luminy, 1992), 77--92,
S\'emin. Congr. 1, Soc. Math. France, Paris, 1996. 

\bibitem{BP}
W.\,Ballmann and P.\,Polymerakis,
Equivariant discretizations of diffusions, random walks, and harmonic functions.
\emph{Enseign. Math.} {\bf 67} (2021), no.~3-4, 331--367.

\bibitem{BP3}
---{}---{}---,
Equivariant discretizations of diffusions and harmonic functions of bounded growth.
\emph{Israel J. Math.} {\bf 249} (2022), no.~1, 227--258. 

\bibitem{BP2}
---{}---{}---,
On the essential spectrum of differential operators over geometrically finite orbifolds.
\emph {J. Differential Geom.} {\bf 127} (2024), no.~1, 1--34.
 
\bibitem{BPe}
---{}---{}---,
Equivariant discretizations of diffusions, random walks, and harmonic functions: Corrections and additions.
\emph{Enseign. Math.} (2025), published online first.
DOI 10.4171/LEM/1091

\bibitem{BHM}
S.\,Blach\`ere, P.\,Ha\"issinsky and P.\,Mathieu,
Harmonic measures versus quasiconformal measures for hyperbolic groups.
\emph{Ann. Sci. \'Ec. Norm. Sup\'er. (4)} {\bf 44} (2011), no.~4, 683--721.

\bibitem{Bo}
B.\,H.\,Bowditch,
Geometrical finiteness with variable negative curvature.
\emph{Duke Math. J.} \textbf{77} (1995), no.~1, 229--274.

\bibitem{Bo1}
---{}---{}---,
Relatively hyperbolic groups.
\emph{Internat. J. Algebra Comput.} \textbf{22} (2012), no.~3, 1250016, 66 pp.

\bibitem{BriHae}
M.\,Bridson and A.\,Haefliger,
\emph{Metric spaces of non-positive curvature}.
Grundlehren der mathematischen Wissenschaften {\bf 319}, Springer-Verlag, Berlin, 1999, xxii+643.

\bibitem{CY}
S.\,Y.\,Cheng and S.-T.\,Yau,
Differential equations on {R}iemannian manifolds and their geometric applications.
\emph{Comm. Pure Appl. Math.} {\bf 28} (1975), no.~3, 333--354.

\bibitem{DS}
C.\,Dru\c{t}u and M.\,Sapir,
Tree-graded spaces and asymptotic cones of groups,
(with an appendix by D. Osin and M. Sapir).
\emph{Topology} {\bf 44} (2005), no.~5, 959--1058.
  
\bibitem{CI}
K.\,Corlette and A.\,Iozzi,
Limit sets of discrete groups of isometries of exotic hyperbolic spaces.
\emph{Trans. Amer. Math. Soc.} {\bf 351} (1999) no.~4, 1507--1530.

\bibitem{GP}
V.\,Gerasimov and L.\,Potyagailo,
Quasi-isometric maps and Floyd boundaries of relatively hyperbolic groups.
\emph{J. Eur. Math. Soc. (JEMS)}  {\bf 15} (2013), no.~6, 2115--2137.

\bibitem{Fu}
H.\,Furstenberg,
\emph{Random walks and discrete subgroups of {L}ie groups}.
Advances in {P}robability and {R}elated {T}opics, vol. 1, 1--63,
Dekker, New York, 1971.

\bibitem{GH}
\'E.\,Ghys and P.\,de la Harpe,
\emph{Sur les groupes hyperboliques d'apr\`es {M}ikhael {G}romov}.
Papers from the Swiss Seminar on Hyperbolic Groups held in Bern, 1988.
Progr. Math. {\bf 83}, Birkh\"auser Boston, Boston, MA, 1990, xii+285.

\bibitem{Go}
W.\,M.\,Goldman,
\emph{Complex hyperbolic geometry}.
Oxford Mathematical Monographs,
Oxford University Press, New York, 1999, xx+316.  

\bibitem{Gou}
S.\,Gou\"ezel, 
Martin boundary of random walks with unbounded jumps in hyperbolic groups.
\emph{Ann. Probab.} {\bf 43} (2015), no.~5, 2374--2404

\bibitem{Gro}
M.\,Gromov,
Hyperbolic groups.
\emph{ Essays in group theory}, 75–263.
Math. Sci. Res. Inst. Publ. 8, Springer-Verlag, New York, 1987

\bibitem{GroMan}
D.\,Groves and J.\,F.\,Manning,
Dehn filling in relatively hyperbolic groups.
\emph{Israel J. Math.} {\bf 168} (2008), 317--429.

\bibitem{Ka}
M.\,Kac,
On the notion of recurrence in discrete stochastic processes.
\emph{Bull. Amer. Math. Soc.} {\bf 53} (1947), 1002--1010.
  
\bibitem{Kai}
V.\,Kaimanovich,
Discretization of bounded harmonic functions on Riemannian manifolds and entropy.
\emph{Potential theory (Nagoya), 1990}, 213--223, de Gruyter, Berlin 1992.

\bibitem{Ki}
Y.\,Kifer,
Brownian motion and positive harmonic functions on complete manifolds of nonpositive curvature.
\emph{From local times to global geometry, control and physics} (Coventry, 1984/85),
Pitman Res. Notes Math. Ser. {\bf 150}, 187--232, Longman Sci. Tech., Harlow, 1986.

\bibitem{KunWata}
H.\,Kunita and T.\,Watanabe,
Markov processes and Martin boundaries.
\emph{Bull. Amer. Math. Soc.} {\bf 69} (1963), 386--391.

\bibitem{KunWatp}
---{}---{}---,
Markov processes and Martin boundaries. I.
\emph{Illinois J. Math.} {\bf 9} (1965), 485--526.

\bibitem{LL}
{F.\,Ledrappier and S.\,Lim},
Local limit theorem in negative curvature.
\emph{Duke Math. J.} {\bf 170} (2021), no.~8, 1585--1681.

\bibitem{LS}
T.\,Lyons and D.\,Sullivan,
Function theory, random paths and covering spaces.
\emph{J. Differential Geom.} {\bf 19} (1984), no.~2, 299--323.

\bibitem{Nan}
D.\,Nandi,
Harmonic measures on Bowditch boundaries of groups hyperbolic relative to virtually nilpotent subgroups.
Arxiv preprint , arXiv:2112.08284.

\bibitem{Nan2}
---{}---{}---,
In preparation.

\bibitem{Osin}
D.V. Osin,
Relatively hyperbolic groups: intrinsic geometry, algebraic properties, and algorithmic problems.
\emph{Mem. Amer. Math. Soc.} \textbf{179} (2006), no.~843, vi+100 pp. 

\bibitem{S-C}
L.\,Saloff-Coste,
A note on Poincar\'e, Sobolev and Harnack inequalities.
\emph{Int. Math. Res. Not.} {\bf 65} (1992), no.~2, 27--38.

\bibitem{Saw}
S.\,A.\,Sawyer, Martin boundaries and random walks.
\emph{Harmonic functions on trees and buildings} (New York, 1995)
17--44, Contemp. Math. 206, Amer. Math. Soc., Providence, RI.

\bibitem{Sch}
B.\,Schapira,
Lemme de l'ombre et non divergence des horosph\`eres d'une vari\'{e}t\'{e} g\'{e}om\'{e}triquement finie.
\emph{Ann. Inst. Fourier (Grenoble)} {\bf 54} (2004), no.~4, 939--987.

\bibitem{Spi}
F.\,L.\,Spitzer,
\emph{Principles of random walk}, second edition.
Graduate Texts in Mathematics, Vol. 34, Springer, New York-Heidelberg, 1976.

\bibitem{SV}
B.\,Stratmann and S.\,L.\,Velani,
The Patterson measure for geometrically finite groups with parabolic elements, new and old.
\emph{Proc. London Math. Soc. (3)} {\bf 71} 1995, no.~1, 197--220.

\bibitem{Su3}
D.\,P.\,Sullivan,
Related aspects of positivity in {R}iemannian geometry.
\emph{J. Differential Geom.} {\bf 25} (1987), no.~3, 327--351.

\bibitem{Y}
A.\,Yaman, 
A topological characterisation of relatively hyperbolic groups.
\emph{J. Reine Angew. Math.} {\bf 566} (2004), 41–89.

\end{thebibliography}
\end{document}